\documentclass[11pt]{amsart}
\usepackage{amssymb}
\usepackage{bbm}
\usepackage{bm}
\usepackage{hyperref}
\hypersetup{
    colorlinks=true, %set true if you want colored links
    linktoc=all,     %set to all if you want both sections and subsections linked
    linkcolor=blue,
        citecolor=red,
    filecolor=black,
    urlcolor=blue	%choose some color if you want links to stand out
}
\usepackage{enumerate}
\usepackage[inline]{enumitem}
\makeatletter
\newcommand{\inlineitem}[1][]{%
\ifnum\enit@type=\tw@
    {\descriptionlabel{#1}}
  \hspace{\labelsep}%
\else
  \ifnum\enit@type=\z@
       \refstepcounter{\@listctr}\fi
    \quad\@itemlabel\hspace{\labelsep}%
\fi} \makeatother
\newcommand{\ga}{\alpha}
\newcommand{\gb}{\beta}
\newcommand{\gga}{\gamma}

\newcommand{\gl}{\lambda}

\newcommand{\gr}{\rho}
\newcommand{\gs}{\sigma}

\newcommand{\Gs}{\Sigma}

\newcommand{\Gom}{\Omega}

\newcommand{\subs}{\subset}

\newcommand{\sbnq}{\subsetneq}

\newcommand{\nsbq}{\nsubseteq}

\newcommand{\bs}{\backslash}

\newcommand{\nin}{\notin}

\newcommand{\ti}{\tilde}

\newcommand{\mbb}{\mathbb}

\newcommand{\mcl}{\mathcal}

\newcommand{\ol}{\overline}

\newcommand{\us}{\underset}
\newcommand{\os}{\overset}

\newcommand{\lra}{\longrightarrow}

\newcommand{\I}{\mcl I}

\newcommand{\N}{\mbb N}
\newcommand{\Z}{\mbb Z}
\newcommand{\R}{\mcl R}

\newcommand{\Ra}{\Rightarrow}

\newcommand{\es}{\emptyset}

\newcommand{\eqdef}{\overset{\mathrm{def}}{=\joinrel=}}

\newcommand{\equ}[1]{%
\begin{equation*}
#1
\end{equation*}
}
\newcommand{\equa}[1]{%
\begin{equation*}
\begin{aligned}
#1
\end{aligned}
\end{equation*}
}
\newcommand{\equan}[2]{%
\begin{equation}
\label{Eq:#1}
\begin{aligned}
#2
\end{aligned}
\end{equation}
}
\DeclareMathOperator{\Det}{Det}

\newcommand{\mattwo}[4]{%
\begin{pmatrix}
  #1 & #2\\ #3 & #4
\end{pmatrix}
}

\newcommand{\matthree}[9]{%
\begin{pmatrix}
  #1 & #2 & #3\\ #4 & #5 & #6\\ #7 & #8 & #9
\end{pmatrix}
}

\theoremstyle{plain}
\newtheorem{theorem}{Theorem}[section]

\newtheorem{prop}[theorem]{Proposition}
\newtheorem{lemma}[theorem]{Lemma}

\newtheorem{ques}[theorem]{Question}

\makeatletter
\def\namedlabel#1#2{\begingroup
	\def\@currentlabel{#2}%
	\label{#1}\endgroup
}
\makeatother
\newtheorem*{thmA}{\bf{Theorem A}}

\newtheorem*{thmOmega}{\bf{Theorem} $\bm{\Gom}$}
\newtheorem*{thmSigma}{\bf{Theorem} $\bm{\Gs}$}
\theoremstyle{definition}
\newtheorem{defn}[theorem]{Definition}

\theoremstyle{remark}
\newtheorem{remark}[theorem]{Remark}
\newtheorem{example}[theorem]{Example}
\numberwithin{equation}{section}
\begin{document}
\title[On the Surjectivity of Certain Maps II]{On the Surjectivity of Certain Maps II: For Generalized Projective Spaces}
\author[C.P. Anil Kumar]{C.P. Anil Kumar}
\address{Flat N0. 104, Bldg. No. 23, Lakshmi Paradise, 5th Main, 11th Cross, Lakshmi Narayana Puram, Near Muneeswaraswamy Narasimhaswamy Temple, Bengaluru-566021, Karnataka, INDIA}
\email{akcp1728@gmail.com}
%\thanks{The author thanks ABCDEFGHIJKLMNOPQRSTUVWXYZ}
\begin{abstract}
In this article we introduce generalized projective spaces (Definitions~[\ref{defn:ProjSpaceRelation},~\ref{defn:GenProjSpace}]) and prove three main theorems in two different
contexts. In the first context we prove, in main Theorem~\ref{theorem:GenCRTSURJ}, the surjectivity of the Chinese remainder reduction map associated to the generalized projective space of an
ideal with a given factorization into mutually co-maximal ideals each of which is contained in only a finitely many maximal ideals,
using the key concept of choice multiplier hypothesis (Definition~\ref{defn:CMH}) which is satisfied.
In the second context of surjectivity of the map from $k\operatorname{-}$dimensional special linear group to the product of generalized projective spaces of $k\operatorname{-}$mutually co-maximal
ideals associating the $k\operatorname{-}$rows or $k\operatorname{-}$columns, we prove remaining two main Theorems~[\ref{theorem:FullGenSurj},~\ref{theorem:FullGenSurjOne}] under certain 
conditions either on the ring or on the generalized projective spaces. Finally in the last section we pose open Questions~[\ref{ques:GenCRTSURJ},~\ref{ques:ProjHighDim}] whose answers
in a greater generality are not known.
\end{abstract}
\maketitle
%%%%%%%%%%%%%%%%%%%%%%%%%%%%%%%%%%%%%%%%%%%%%%%%%%%%%%%%%%%%%%%%%%%%%%%%%%%%%%%%%%%%%%%%%
%%%%%%%%%%%%%%%%%%%%%%%%%%%%%%%%%%%%%%%%%%%%%%%%%%%%%%%%%%%%%%%%%%%%%%%%%%%%%%%%%%%%%%%%%
%%%%%%%%%%%%%%%%%%%%%%%%%%%%%%%%%%%%%%%%%%%%%%%%%%%%%%%%%%%%%%%%%%%%%%%%%%%%%%%%%%%%%%%%%

\section{\bf{Introduction}}
In this article we are concerned with surjectivity of two maps in two different contexts. This is the sequel to C.~P.~Anil~Kumar~\cite{CPAK}. 
The first context is regarding the Chinese remainder reduction map. Ordinary Chinese Remainder Theorem of number theory has been generalized in various contexts. We quote a couple of them here. 
\begin{enumerate}[label=(\Alph*)]
\item We have a formulation for ideals in commutative ring theory. 
If $\R$ is a commutative ring with unity and $\mcl{I}_1,\mcl{I}_2,\ldots,\mcl{I}_k$ are mutually co-maximal ideals then we the map \equ{\frac{\R}{\us{i=1}{\os{k}{\cap}}\I_i} \lra \us{i=1}{\os{k}{\prod}} \frac{\R}{\mcl{I}_i}}
is surjective. 
\item We have a formulation in the form of approximation theorems for algebraic groups and arbitrary varieties. 
For example the reduction mod $m$ map \equ{\gr_m:SL_2(\Z) \lra SL_2(\Z/m\Z)} is surjective for any $m>1$.
Suppose $X\subs \mbb{A}^d_{\Z}$ is defined to be the zero set of a family $f_{\ga},\ga \in I$ of polynomials in $d\operatorname{-}$variables with integers coefficients. Then we can define for any $m\in \mbb{N}$,
$X(\Z/m\Z)=\{(a_1,a_2,\ldots,a_d)\in (\Z/m\Z)^d \mid f_{\ga}(a_1,a_2,\ldots,a_d)\equiv 0\mod\ m \text{ for all }\ga \in I\}$. 
Then a formulation of the approximation question asks whether these maps are surjective for all $m$. More about strong approximation can be found in A.~S.~Rapinchuk~\cite{SA}. 
\end{enumerate}
We consider the generalized projective spaces associated to ideals (refer to Definitions~[\ref{defn:ProjSpaceRelation},~\ref{defn:GenProjSpace}]) of any positive integer dimension $k$ over a commutative ring $\R$
with unity obtained from unital $(k+1)\operatorname{-}$vectors in $\R^{k+1}$. We look at a formulation of Chinese Remainder Theorem for these projective spaces and prove main Theorem~\ref{theorem:GenCRTSURJ}.
This generalizes Theorem $1.6$ on Page $338$ in C.~P.~Anil~Kumar~\cite{CPAK}. In higher generality, Question~\ref{ques:GenCRTSURJ} is still open.

In the second context we look at the following question for generalized projective spaces. Let $k\in \mbb{N}$. Given $(k+1)$ elements in $(k+1)$ generalized projective spaces associated to $(k+1)$ ideals when does 
there exist an $SL_{k+1}(\R)$ matrix such that the rows give rise to these elements. This question is answered in the affirmative for Dedekind type domains $\R$ (Refer to the third main result, Theorem $1.8$ on Page $339$)
in C.~P.~Anil~Kumar~\cite{CPAK} for ordinary projective spaces. Here we generalize further and prove next couple of main Theorems~[\ref{theorem:FullGenSurj},~\ref{theorem:FullGenSurjOne}]. These main theorems lead to another 
open Question~\ref{ques:ProjHighDim} in greater generality. Appropriate and necessary examples are given after the statement of main theorems as motivation.

%%%%%%%%%%%%%%%%%%%%%%%%%%%%%%%%%%%%%%%%%%%%%%%%%%%%%%%%%%%%%%%%%%%%%%%%%%%%%%%%%%%%%%%%%
%%%%%%%%%%%%%%%%%%%%%%%%%%%%%%%%%%%%%%%%%%%%%%%%%%%%%%%%%%%%%%%%%%%%%%%%%%%%%%%%%%%%%%%%%
%%%%%%%%%%%%%%%%%%%%%%%%%%%%%%%%%%%%%%%%%%%%%%%%%%%%%%%%%%%%%%%%%%%%%%%%%%%%%%%%%%%%%%%%%

\section{\bf{The main results}}
In this section we state the main results. Before stating the main results we need to introduce the definition of generalized projective spaces and one more standard definition.

%%%%%%%%%%%%%%%%%%%%%%%%%%%%%%%%%%%%%%%%%%%%%%%%%%%%%%%%%%%%%%%%%%%%%%%%%%%%%%%%%%%%%%%%%
%%%%%%%%%%%%%%%%%%%%%%%%%%%%%%%%%%%%%%%%%%%%%%%%%%%%%%%%%%%%%%%%%%%%%%%%%%%%%%%%%%%%%%%%%
%%%%%%%%%%%%%%%%%%%%%%%%%%%%%%%%%%%%%%%%%%%%%%%%%%%%%%%%%%%%%%%%%%%%%%%%%%%%%%%%%%%%%%%%%

\begin{defn}[Definition of a Projective Space Relation]
\label{defn:ProjSpaceRelation}
~\\
Let $\R$ be a commutative ring with unity. Let $k\in \mbb{N}$ and \equ{\mcl{GCD}_{k+1}(\R)=\{(a_0,a_1,a_2,\ldots,a_k)\in \R^{k+1}\mid \us{i=0}{\os{k}{\sum}}(a_i)=\R\}.} Let $\I \subsetneq \R$ be an ideal
and $m_0,m_1,\ldots,m_k\in \mbb{N}$. Define an equivalence relation \equ{\sim^{k,(m_0,m_1,\ldots,m_k)}_{\I}} on $\mcl{GCD}_{k+1}(\R)$
as follows. For \equ{(a_0,a_1,a_2,\ldots,a_k),(b_0,b_1,b_2,\ldots,b_k) \in \mcl{GCD}_{k+1}(\R)} we say \equ{(a_0,a_1,a_2,\ldots,a_k)\sim^{\{k,(m_0,m_1,\ldots,m_k)\}}_{\I}(b_0,b_1,b_2,\ldots,b_k)} if there exists a \equ{\gl\in \R\text{ with }\ol{\gl}\in \bigg(\frac{\R}{\I}\bigg)^{*}} such that we have 
\equ{a_i \equiv \gl^{m_i}b_i \mod \I,0\leq i\leq k.}  
\end{defn}

Lemma~\ref{lemma:EquivObs} proves that $\sim^{k,(m_0,m_1,\ldots,m_k)}_{\I}$ is an equivalence relation for any ideal $\I \subsetneq \R$.  We mention an example now.

\begin{example}
We take in this example $\R=\Z,\mcl{I}=(p)$ where $p$ is a prime, $k=1,m_0=1,m_1=2$. Consider the equivalence relation $\sim^{1,(1,2)}_{(p)}$. Here the equivalence classes are given as follows.
	\begin{itemize}
		\item One equivalence class: $E_0=\{(a,pb)\mid a,b\in \Z,p\nmid a,gcd(a,pb)=1 \}$ which modulo $p$ has size $(p-1)$.
		\item $(p-1)$ equivalence classes: For $1 \leq b\leq p-1$ we have $E_b=\{(\gl+pa,\gl^2b+pc)\mid a,b,c,\gl \in \Z,p\nmid \gl,gcd(\gl+pa,\gl^2b+pc)=1\}$ each of which modulo $p$ has size $p-1$.
		\item One equivalence class: $R_1=\{(pa,\gl^2+pb)\mid a,b,\gl\in \Z, p\nmid \gl,gcd(pa,\gl^2+pb)=1\}$ which modulo $p$ has size $\frac{p-1}{2}$.
		\item One equivalence class: Let $n\in \Z$ be a fixed quadratic non-residue modulo $p$. Then we have $O_n=\{(pa,\gl^2n+pb)\mid a,b,\gl\in \Z, p\nmid \gl, gcd(pa,\gl^2n+pb)=1\}$ which modulo $p$ has size $\frac{p-1}{2}$.
	\end{itemize}
	The total number of equivalence classes is $p+2$ which is one more than the number of elements in the usual projective space $\mbb{PF}^1_p$.
	A set of representatives for the equivalence classes  is given by
	\begin{itemize}
		\item $(1,b)\in E_b, 0\leq b\leq p-1$.
		\item $(0,1)\in R_1$
		\item $(0,n)\in O_n$ where $n$ is a chosen quadratic non-residue modulo $p$.
	\end{itemize}
\end{example}

\begin{example}
We take in this example $\R=\mbb{R},\mcl{I}=(0),k\in \N,m_i=2,0\leq i\leq k$.
In this case the equivalence relation $\sim^{k,(2,2,\ldots,2)}_{(0)}$ is such that the equivalence classes are rays emanating from origin in $\mbb{R}^{k+1}$.	
\end{example}

\begin{example}
We take in this example $\R=\mbb{F}_p,p$ a prime, $\mcl{I}=(0),k\in \N,m_i=2,0\leq i\leq k$.
In this case the equivalence relation $\sim^{k,(2,2,\ldots,2)}_{(0)}$ is such that the equivalence classes are rays emanating from origin in $\mbb{F}_p^{k+1}$ where a ray is defined as follows. For $v=(a_0,a_1,\ldots,a_k)\in \mbb{F}_p^{k+1}\bs \{0\}$ the ray is given by $R_{v}=\{\gl^2(a_0,a_1,\ldots,a_k)\mid \gl\in \mbb{F}_p^{*}\}$. There are more equivalence classes than just the number of lines passing through origin.	
\end{example}
%%%%%%%%%%%%%%%%%%%%%%%%%%%%%%%%%%%%%%%%%%%%%%%%%%%%%%%%%%%%%%%%%%%%%%%%%%%%%%%%%%%%%%%%%
%%%%%%%%%%%%%%%%%%%%%%%%%%%%%%%%%%%%%%%%%%%%%%%%%%%%%%%%%%%%%%%%%%%%%%%%%%%%%%%%%%%%%%%%%
%%%%%%%%%%%%%%%%%%%%%%%%%%%%%%%%%%%%%%%%%%%%%%%%%%%%%%%%%%%%%%%%%%%%%%%%%%%%%%%%%%%%%%%%%

We define the generalized projective space associated to an ideal in a commutative ring with unity.
\begin{defn}
\label{defn:GenProjSpace}
Let $\R$ be a commutative ring with unity. Let $k\in \mbb{N}$ and \equ{\mcl{GCD}_{k+1}(\R)=\{(a_0,a_1,a_2,\ldots,a_k)\in \R^{k+1}\mid \us{i=0}{\os{k}{\sum}}(a_i)=\R\}.}
Let $m_0,m_1,\ldots,m_k\in \mbb{N}$ and $\I \subsetneq \R$ be an ideal. Let $\sim^{k,(m_0,m_1,\ldots,m_k)}_{\I}$ denote the equivalence relation as in Definition~\ref{defn:ProjSpaceRelation}. Then we define 
\equ{\mbb{PF}^{k,(m_0,m_1,\ldots,m_k)}_{\I}\eqdef \frac{\mcl{GCD}_{k+1}(\R)}{\sim^{k,(m_0,m_1,\ldots,m_k)}_{\I}}.}
If $\mcl{I}=\R$ then let $\sim^{k,(m_0,m_1,\ldots,m_k)}_{\I}$ be the trivial equivalence relation on $\mcl{GCD}_{k+1}(\R)$ where any two elements are related.
We define \equ{\mbb{PF}^{k,(m_0,m_1,\ldots,m_k)}_{\I} \eqdef \frac{\mcl{GCD}_{k+1}(\R)}{\sim^{k,(m_0,m_1,\ldots,m_k)}_{\I}}}
a singleton set having single equivalence class.
\end{defn}

\begin{remark}
Let $\mbb{K}$ be a field. Then by choosing $\R=\mbb{K},\mcl{I}=(0),k\in \N,m_i=1,0\leq i\leq k$ we get the projective space $\mbb{PF}^{k,(1,1,\ldots,1)}_{(0)}$ which is the usual $k$-dimensional projective space $\mbb{PF}^k_{\mbb{K}}$ over the field $\mbb{K}$.
\end{remark}

Lemma~\ref{lemma:GenEquivRel} proves that Definition~\ref{defn:GenProjSpace} is one valid generalization of the usual projective space associated to an ideal $\mcl{I}$ which can be expressed as a finite product of ideals whose radicals are distinct maximal ideals.

%%%%%%%%%%%%%%%%%%%%%%%%%%%%%%%%%%%%%%%%%%%%%%%%%%%%%%%%%%%%%%%%%%%%%%%%%%%%%%%%%%%%%%%%%
%%%%%%%%%%%%%%%%%%%%%%%%%%%%%%%%%%%%%%%%%%%%%%%%%%%%%%%%%%%%%%%%%%%%%%%%%%%%%%%%%%%%%%%%%
%%%%%%%%%%%%%%%%%%%%%%%%%%%%%%%%%%%%%%%%%%%%%%%%%%%%%%%%%%%%%%%%%%%%%%%%%%%%%%%%%%%%%%%%%

We define a Jacobson element, a non-Jacobson element and a Jacobson ideal.
\begin{defn}
Let $\R$ be a commutative ring with unity. An element is said to be a Jacobson element if it is contained in the Jacobson radical $\mcl{J}(\R) = \us{\mcl{M} \subs \R,\text{ maximal ideal}}{\bigcap} \mcl{M}$.
An ideal is a Jacobson ideal if it is contained in the Jacobson radical $\mcl{J}(\R)$. An element is said to be non-Jacobson element if it is not in the Jacobson radical $\mcl{J}(\R)$ or equivalently 
it is a unit modulo some maximal ideal. In particular if an element is a unit modulo a proper ideal of the ring $\R$ then it is a non-Jacobson element.
\end{defn}

%%%%%%%%%%%%%%%%%%%%%%%%%%%%%%%%%%%%%%%%%%%%%%%%%%%%%%%%%%%%%%%%%%%%%%%%%%%%%%%%%%%%%%%%%
%%%%%%%%%%%%%%%%%%%%%%%%%%%%%%%%%%%%%%%%%%%%%%%%%%%%%%%%%%%%%%%%%%%%%%%%%%%%%%%%%%%%%%%%%
%%%%%%%%%%%%%%%%%%%%%%%%%%%%%%%%%%%%%%%%%%%%%%%%%%%%%%%%%%%%%%%%%%%%%%%%%%%%%%%%%%%%%%%%%

\subsection{\bf{The first main theorem}}
The first main theorem is stated as follows where no condition on the ring is required:

%%%%%%%%%%%%%%%%%%%%%%%%%%%%%%%%%%%%%%%%%%%%%%%%%%%%%%%%%%%%%%%%%%%%%%%%%%%%%%%%%%%%%%%%%
%%%%%%%%%%%%%%%%%%%%%%%%%%%%%%%%%%%%%%%%%%%%%%%%%%%%%%%%%%%%%%%%%%%%%%%%%%%%%%%%%%%%%%%%%
%%%%%%%%%%%%%%%%%%%%%%%%%%%%%%%%%%%%%%%%%%%%%%%%%%%%%%%%%%%%%%%%%%%%%%%%%%%%%%%%%%%%%%%%%

\begin{thmA}
\namedlabel{theorem:GenCRTSURJ}{A}
Let $\R$ be a commutative ring with unity and $k,l\in \mbb{N}$. Let $\mcl{I}_i,1\leq i\leq k$ be mutually co-maximal ideals each of which is contained in only a finitely many maximal ideals
and $\mcl{I}=\us{i=1}{\os{k}{\prod}}\mcl{I}_i$. Let $m_j\in \mbb{N},0\leq j\leq l$. Then the Chinese remainder reduction map associated to the projective space 
\equ{\mbb{PF}^{l,(m_0,m_1,\ldots,m_l)}_{\mcl{I}} \lra \mbb{PF}^{l,(m_0,m_1,\ldots,m_l)}_{\mcl{I}_1} \times \mbb{PF}^{l,(m_0,m_1,\ldots,m_l)}_{\mcl{I}_2} \times \ldots \times \mbb{PF}^{l,(m_0,m_1,\ldots,m_l)}_{\mcl{I}_k}}
is surjective (in fact bijective).

In particular let $\mathit{M}(\R)=\{\I \subs \R\mid \I =\mcl{Q}_1\mcl{Q}_2\ldots \mcl{Q}_s$ 
with rad$(\mcl{Q}_i)=\mcl{M}_i, 1\leq i\leq s$  are maximal ideals in $\R$ where $\mcl{M}_i\neq \mcl{M}_j, 1\leq i\neq j\leq s \}$. 
Let \equ{\mcl{I}_i=\mcl{Q}_{i1}\mcl{Q}_{i2}\ldots \mcl{Q}_{ir_i} \in \mathit{M}(\R)\text{ or } \mcl{I}_i=\R, 0\leq i\leq k} be $(k+1)\operatorname{-}$pairwise co-maximal ideals in $\R$
where rad$(\mcl{Q}_{ij})=\mcl{M}_{ij}, 1\leq j\leq r_i, 0\leq i\leq k$ are distinct maximal ideals in $\R$. Let $\mcl{I}=\us{i=1}{\os{k}{\prod}}\mcl{I}_i$.
Then the Chinese remainder reduction map associated to the projective space 
\equ{\mbb{PF}^{l,(m_0,m_1,\ldots,m_l)}_{\mcl{I}} \lra 
\mbb{PF}^{l,(m_0,m_1,\ldots,m_l)}_{\mcl{I}_1} \times \mbb{PF}^{l,(m_0,m_1,\ldots,m_l)}_{\mcl{I}_2} \times \ldots \times \mbb{PF}^{l,(m_0,m_1,\ldots,m_l)}_{\mcl{I}_k}}
is surjective (in fact bijective).
\end{thmA}

%%%%%%%%%%%%%%%%%%%%%%%%%%%%%%%%%%%%%%%%%%%%%%%%%%%%%%%%%%%%%%%%%%%%%%%%%%%%%%%%%%%%%%%%%
%%%%%%%%%%%%%%%%%%%%%%%%%%%%%%%%%%%%%%%%%%%%%%%%%%%%%%%%%%%%%%%%%%%%%%%%%%%%%%%%%%%%%%%%%
%%%%%%%%%%%%%%%%%%%%%%%%%%%%%%%%%%%%%%%%%%%%%%%%%%%%%%%%%%%%%%%%%%%%%%%%%%%%%%%%%%%%%%%%%

\subsection{\bf{The second main theorem}}
The second main theorem is stated as follows where a condition on the ring is required:

%%%%%%%%%%%%%%%%%%%%%%%%%%%%%%%%%%%%%%%%%%%%%%%%%%%%%%%%%%%%%%%%%%%%%%%%%%%%%%%%%%%%%%%%%
%%%%%%%%%%%%%%%%%%%%%%%%%%%%%%%%%%%%%%%%%%%%%%%%%%%%%%%%%%%%%%%%%%%%%%%%%%%%%%%%%%%%%%%%%
%%%%%%%%%%%%%%%%%%%%%%%%%%%%%%%%%%%%%%%%%%%%%%%%%%%%%%%%%%%%%%%%%%%%%%%%%%%%%%%%%%%%%%%%%

\begin{thmOmega}
\namedlabel{theorem:FullGenSurj}{$\Gom$}
Let $\R$ be a commutative ring with unity. Suppose every non-Jacobson element in $\R$ is contained in only a finitely many maximal ideals.
Let $k\in \mbb{N}$ and $\mcl{I}_i,0\leq i\leq k$ be mutually co-maximal ideals in $\R$.
Also if there is exactly one proper ideal $\mcl{I}_j$ for some $0\leq j\leq k$ then we suppose it is contained in only a finitely many maximal ideals.
Let $m_j^i\in \mbb{N}, 0\leq i,j\leq k$. 
Then the map 
\equ{SL_{k+1}(\R) \lra \us{i=0}{\os{k}{\prod}}\mbb{PF}^{k,(m^i_0,m^i_1,\ldots,m^i_k)}_{\I_i}} given by
\equa{&A_{(k+1)\times (k+1)}=[a_{i,j}]_{0\leq i,j\leq k} \lra\\
&\big([a_{0,0}:a_{0,1}:\ldots: a_{0,k}],[a_{1,0}:a_{1,1}:\ldots: a_{1,k}],\ldots,[a_{k,0}:a_{k,1}:\ldots: a_{k,k}]\big)}
is surjective. 

In particular let $\mathit{M}(\R)=\{\I \subs \R\mid \I =\mcl{Q}_1\mcl{Q}_2\ldots \mcl{Q}_l$ 
with rad$(\mcl{Q}_i)=\mcl{M}_i, 1\leq i\leq l$  are maximal ideals in $\R$ where $\mcl{M}_i\neq \mcl{M}_j, 1\leq i\neq j\leq l\}$. 
Let \equ{\mcl{I}_i=\mcl{Q}_{i1}\mcl{Q}_{i2}\ldots \mcl{Q}_{ir_i} \in \mathit{M}(\R)\text{ or } \mcl{I}_i=\R, 0\leq i\leq k} be $(k+1)\operatorname{-}$pairwise co-maximal ideals in $\R$
where rad$(\mcl{Q}_{ij})=\mcl{M}_{ij}, 1\leq j\leq r_i, 0\leq i\leq k$ are distinct maximal ideals in $\R$.
Then the map 
\equ{SL_{k+1}(\R) \lra \us{i=0}{\os{k}{\prod}}\mbb{PF}^{k,(m^i_0,m^i_1,\ldots,m^i_k)}_{\I_i}} given by
\equa{&A_{(k+1)\times (k+1)}=[a_{i,j}]_{0\leq i,j\leq k} \lra\\
&\big([a_{0,0}:a_{0,1}:\ldots: a_{0,k}],[a_{1,0}:a_{1,1}:\ldots: a_{1,k}],\ldots,[a_{k,0}:a_{k,1}:\ldots: a_{k,k}]\big)}
is surjective. 
\end{thmOmega}
The hypothesis about exactly one proper ideal in Theorem~\ref{theorem:FullGenSurj} is to make sure that each of the ideals is contained in only a finitely many maximal ideals by mutual co-maximality.
When the ring has infinitely many maximal ideals, this hypothesis avoids the case where we have one Jacobson ideal and the rest are all unit ideals.
We will mention about this hypothesis later in this article as it appears again at a few places.

%%%%%%%%%%%%%%%%%%%%%%%%%%%%%%%%%%%%%%%%%%%%%%%%%%%%%%%%%%%%%%%%%%%%%%%%%%%%%%%%%%%%%%%%%
%%%%%%%%%%%%%%%%%%%%%%%%%%%%%%%%%%%%%%%%%%%%%%%%%%%%%%%%%%%%%%%%%%%%%%%%%%%%%%%%%%%%%%%%%
%%%%%%%%%%%%%%%%%%%%%%%%%%%%%%%%%%%%%%%%%%%%%%%%%%%%%%%%%%%%%%%%%%%%%%%%%%%%%%%%%%%%%%%%%

\subsection{\bf{The third main theorem}}

The third main theorem is stated as follows where no condition on the ring is required but a condition on the generalized projective space is assumed:

%%%%%%%%%%%%%%%%%%%%%%%%%%%%%%%%%%%%%%%%%%%%%%%%%%%%%%%%%%%%%%%%%%%%%%%%%%%%%%%%%%%%%%%%%
%%%%%%%%%%%%%%%%%%%%%%%%%%%%%%%%%%%%%%%%%%%%%%%%%%%%%%%%%%%%%%%%%%%%%%%%%%%%%%%%%%%%%%%%%
%%%%%%%%%%%%%%%%%%%%%%%%%%%%%%%%%%%%%%%%%%%%%%%%%%%%%%%%%%%%%%%%%%%%%%%%%%%%%%%%%%%%%%%%%

\begin{thmSigma}
\namedlabel{theorem:FullGenSurjOne}{$\Gs$}
Let $\R$ be a commutative ring with unity. Let $k\in \mbb{N}$ and $\mcl{I}_i,0\leq i\leq k$ be mutually co-maximal ideals in $\R$ each of which is contained in only a finitely many maximal ideals. 
Let $m_j^i\in \mbb{N}, 0\leq i,j\leq k$ such that $m^i_{\gs(i)}=1$ for each $0\leq i\leq k$ for some $\gs$, a permutation of $\{0,1,2,\ldots,k\}$.  
Then the map 
\equ{SL_{k+1}(\R) \lra \us{i=0}{\os{k}{\prod}}\mbb{PF}^{k,(m^i_0,m^i_1,\ldots,m^i_k)}_{\I_i}} given by
\equa{&A_{(k+1)\times (k+1)}=[a_{i,j}]_{0\leq i,j\leq k} \lra\\
&\big([a_{0,0}:a_{0,1}:\ldots: a_{0,k}],[a_{1,0}:a_{1,1}:\ldots: a_{1,k}],\ldots,[a_{k,0}:a_{k,1}:\ldots: a_{k,k}]\big)}
is surjective. Also we have the following particular instances.
\begin{enumerate}
\item 
Especially if $m^i_j=1$ for all $0\leq i,j\leq k$, that is, for usual projective spaces, also, we have surjectivity.
\item 
Also in particular let $\mathit{M}(\R)=\{\I \subs \R\mid \I =\mcl{Q}_1\mcl{Q}_2\ldots \mcl{Q}_l$ 
with rad$(\mcl{Q}_i)=\mcl{M}_i, 1\leq i\leq l$  are maximal ideals in $\R$ where $\mcl{M}_i\neq \mcl{M}_j, 1\leq i\neq j\leq l\}$. 
Let \equ{\mcl{I}_i=\mcl{Q}_{i1}\mcl{Q}_{i2}\ldots \mcl{Q}_{ir_i} \in \mathit{M}(\R)\text{ or } \mcl{I}_i=\R, 0\leq i\leq k} be $(k+1)\operatorname{-}$pairwise co-maximal ideals in $\R$
where rad$(\mcl{Q}_{ij})=\mcl{M}_{ij}, 1\leq j\leq r_i, 0\leq i\leq k$ are distinct maximal ideals in $\R$.
Let $m_j^i\in \mbb{N}, 0\leq i,j\leq k$ such that $m^i_{\gs(i)}=1$ for each $0\leq i\leq k$ for some $\gs$, a permutation of $\{0,1,2,\ldots,k\}$.  
Then the map 
\equ{SL_{k+1}(\R) \lra \us{i=0}{\os{k}{\prod}}\mbb{PF}^{k,(m^i_0,m^i_1,\ldots,m^i_k)}_{\I_i}} given by
\equa{&A_{(k+1)\times (k+1)}=[a_{i,j}]_{0\leq i,j\leq k} \lra\\
&\big([a_{0,0}:a_{0,1}:\ldots: a_{0,k}],[a_{1,0}:a_{1,1}:\ldots: a_{1,k}],\ldots,[a_{k,0}:a_{k,1}:\ldots: a_{k,k}]\big)}
is surjective. 
\end{enumerate}
\end{thmSigma}
We give appropriate motivating examples which are very interesting and relevant for this aricle.  First we mention an example as an application of Theorem~\ref{theorem:FullGenSurj} but Theorem~\ref{theorem:FullGenSurjOne} is not applicable here.
\begin{example}
	Let $\R=\Z$. Let $k=4$ and $\mcl{I}_0=241\Z,\mcl{I}_1=601\Z,\mcl{I}_2=1201\Z,\mcl{I}_3=1321\Z,\mcl{I}_4=1801\Z$ be the five distinct primes ideals in $\Z$.
	Let \equ{M=[m^i_j]_{0\leq i,j\leq 4}=\begin{pmatrix}
			2 & 5 & 3 & 10 & 6\\
			8 & 20 & 30 & 24 & 12\\
			1 & 50 & 48 & 40 & 60\\
			11 & 55 & 44 & 22 & 15\\
			18 & 4 & 72 & 90 & 27\\
		\end{pmatrix}.}
	
	Consider the row unital matrix say for example 
	\equ{B=[b_{ij}]_{0\leq i,j\leq 4}=\begin{pmatrix}
			1 & 1 & 1 & 1 & 1\\
			1 & 1 & 1 & 1 & 1\\
			1 & 1 & 1 & 1 & 1\\
			1 & 1 & 1 & 1 & 1\\
			1 & 1 & 1 & 1 & 1\\
	\end{pmatrix}}
	which gives rise to the element  
	\equa{&([1:1:1:1:1],[1:1:1:1:1],[1:1:1:1:1],[1:1:1:1:1],[1:1:1:1:1])\in\\ &\mbb{PF}_{(241)}^{4,(2,5,3,10,6)}\times \mbb{PF}_{(601)}^{4,(8,20,30,24,12)}\times \mbb{PF}_{(1201)}^{4,(1,50,48,40,60)}\times \mbb{PF}_{(1321)}^{4,(11,55,44,22,15)}\times \mbb{PF}_{(1801)}^{4,(18,4,72,90,27)}.}
	
	So does there exist a matrix $A=[a_{ij}]_{0\leq i,j\leq 4}\in SL_5(\Z)$ such that $a_{ij}\equiv \gl_i^{m^i_j}b_{ij} \mod \mcl{I}_i, 0\leq i,j\leq 4,\gl_i\in \Z$ such that $\ol{\gl_i}\in \frac{\Z}{\mcl{I}_i}$ is a unit? The answer is yes and there does exist such a matrix $A\in SL_5(\Z)$ using Theorem~\ref{theorem:FullGenSurj}.
	In this example every non-jacobson element in $\Z$ is contained in only a finitely many maximal ideals. Theorem~\ref{theorem:FullGenSurjOne} is not applicable because of the values in  $M=[m^i_j]_{0\leq i,j\leq 4}$.	
\end{example}

We mention an example where Theorem~\ref{theorem:FullGenSurj} is not applicable but Theorem~\ref{theorem:FullGenSurjOne} is applicable.

\begin{example}
	Let $\R=\mbb{R}[x,y]$. Let $k=2,r,s,t\in \N$. Let $\mcl{I}_0=(x,y)^t, \mcl{I}_1=(x-1,y)^r,\mcl{I}_2=(x,y-1)^s$.
	$M=[m^i_j]_{0\leq i,j\leq 2}=\matthree{1}{2}{2}{2}{1}{2}{2}{2}{1}$. So here by Theorem~\ref{theorem:FullGenSurjOne} the map $SL_3(\R)\lra \mbb{PF}^{2,(1,2,2)}_{\mcl{I}_0}\times \mbb{PF}^{2,(2,1,2)}_{\mcl{I}_1} \times \mbb{PF}^{2,(2,2,1)}_{\mcl{I}_2}$ is surjective. It is not true that every non-jacobson element of $\R$ is contained in only a finitely many maximal ideals. So Theorem~\ref{theorem:FullGenSurj} is not applicable.
\end{example}

Now we mention an example where Theorems~\ref{theorem:FullGenSurjOne},~\ref{theorem:FullGenSurj} need not hold if proper conditions are not given.
\begin{example}
	\label{Example:NotSurj}
	Consider the ring $\R=\frac{\mbb{R}[X_1,X_2,X_3]}{(X_1^2+X_2^2+X_3^2-1)}=\mbb{R}[x_1=\ol{X_1},x_2=\ol{X_2},x_3=\ol{X_3}]$. This ring does not satisfy unimodular extension property. Note the vector $(x_1,x_2,x_3)$ is unimodular. However there does not exist a matrix in $GL_3(\R)$ such that the first row is $(x_1,x_2,x_3)$. This is because, there is no non-vanishing tangent vector field over the two dimensional sphere $S^2$. Let $\mcl{I}_0=(0),\mcl{I}_1=\R,\mcl{I}_2=\R, m^i_j=1, 0\leq i,j\leq 2,k=2$. Here the zero ideal $\mcl{I}_0$ is not contained in finitely many maximal ideals. The map $SL_3(\R) \lra \mbb{PF}_{\mcl{I}_0}^{2,(1,1,1)} \times \mbb{PF}_{\mcl{I}_1}^{2,(1,1,1)} \times \mbb{PF}_{\mcl{I}_2}^{2,(1,1,1)}$ is not surjective as the element $([x_1:x_2:x_3],*,*)\in \mbb{PF}_{\mcl{I}_0}^{2,(1,1,1)} \times \mbb{PF}_{\mcl{I}_1}^{2,(1,1,1)} \times \mbb{PF}_{\mcl{I}_2}^{2,(1,1,1)}$ is not in the image. 	
\end{example}

We give an example where both Theorems~\ref{theorem:FullGenSurj},~\ref{theorem:FullGenSurjOne} are not applicable however Theorem~\ref{theorem:FullGenSurjTwoDim} is applicable. 

\begin{example}
Let $\R=\mbb{R}[x,y]$. Let $k=1,r,s\in \N$. Let $\mcl{I}_1=(x-1,y)^r,\mcl{I}_2=(x,y-1)^s$.
$m^i_j=2,i=1,2,j=0,1$. Theorem~\ref{theorem:FullGenSurj} is not applicable here. Also, because of the values of $m^i_j,i=1,2,j=0,1$, Theorem~\ref{theorem:FullGenSurjOne} is not applicable. However Theorem~\ref{theorem:FullGenSurjTwoDim} is applicable as the two ideals are co-maximal and each one of them is contained in only one maximal ideal, that is, $\mcl{I}_1 \subseteq (x-1,y),\mcl{I}_2\subseteq (x,y-1)$. So the map 
\equ{SL_2(\R) \lra \mbb{PF}^{1,(m^1_0,m^1_1)}_{\I_1} \times \mbb{PF}^{1,(m^2_0,m^2_1)}_{\I_2}} given by
\equ{\mattwo abcd \lra ([a:b],[c:d])}
is surjective.
\end{example}

We mention an example which leads to open Question~\ref{ques:ProjHighDim} in Section~\ref{sec:OQ}. 
\begin{example}
	\label{Example:OQ}
Let $\mcl{R}=\mbb{R}[x,y]$. Let $k=2,r,s,t\in \N$. Let $\mcl{I}_0=(x,y)^t, \mcl{I}_1=(x-1,y)^r,\mcl{I}_2=(x,y-1)^s$.
$M=[m^i_j]_{0\leq i,j\leq 2}=\matthree{2}{2}{2}{2}{2}{2}{2}{2}{2}$. Here Theorems~\ref{theorem:FullGenSurj},~\ref{theorem:FullGenSurjOne},~\ref{theorem:FullGenSurjTwoDim} are not applicable. 
\end{example}

%%%%%%%%%%%%%%%%%%%%%%%%%%%%%%%%%%%%%%%%%%%%%%%%%%%%%%%%%%%%%%%%%%%%%%%%%%%%%%%%%%%%%%%%%
%%%%%%%%%%%%%%%%%%%%%%%%%%%%%%%%%%%%%%%%%%%%%%%%%%%%%%%%%%%%%%%%%%%%%%%%%%%%%%%%%%%%%%%%%
%%%%%%%%%%%%%%%%%%%%%%%%%%%%%%%%%%%%%%%%%%%%%%%%%%%%%%%%%%%%%%%%%%%%%%%%%%%%%%%%%%%%%%%%%

\section{\bf{Generalized Projective Spaces Over Commutative Rings with Unity}}
In this section we prove in Lemma~\ref{lemma:GenEquivRel} that the generalized projective space is a valid generalization of the usual projective space associated to any ideal in a commutative ring with unity.
We observe that the relation in Definition~\ref{defn:ProjSpaceRelation} is an equivalence relation.

%%%%%%%%%%%%%%%%%%%%%%%%%%%%%%%%%%%%%%%%%%%%%%%%%%%%%%%%%%%%%%%%%%%%%%%%%%%%%%%%%%%%%%%%%
%%%%%%%%%%%%%%%%%%%%%%%%%%%%%%%%%%%%%%%%%%%%%%%%%%%%%%%%%%%%%%%%%%%%%%%%%%%%%%%%%%%%%%%%%
%%%%%%%%%%%%%%%%%%%%%%%%%%%%%%%%%%%%%%%%%%%%%%%%%%%%%%%%%%%%%%%%%%%%%%%%%%%%%%%%%%%%%%%%%

\begin{lemma}
\label{lemma:EquivObs}
In Definition~\ref{defn:ProjSpaceRelation} the relation \equ{\sim^{k,(m_0,m_1,\ldots,m_k)}_{\I}} is an equivalence relation on the set $\mcl{GCD}_{k+1}(\R)$. 
\end{lemma}
\begin{proof}
The proof is immediate.
\end{proof}

%%%%%%%%%%%%%%%%%%%%%%%%%%%%%%%%%%%%%%%%%%%%%%%%%%%%%%%%%%%%%%%%%%%%%%%%%%%%%%%%%%%%%%%%%
%%%%%%%%%%%%%%%%%%%%%%%%%%%%%%%%%%%%%%%%%%%%%%%%%%%%%%%%%%%%%%%%%%%%%%%%%%%%%%%%%%%%%%%%%
%%%%%%%%%%%%%%%%%%%%%%%%%%%%%%%%%%%%%%%%%%%%%%%%%%%%%%%%%%%%%%%%%%%%%%%%%%%%%%%%%%%%%%%%%

We observe that the equivalence relation $\sim^{k,(m_0,m_1,\ldots,m_k)}_{\I}$ generalizes the usual equivalence relation in the definition of usual projective space. We prove 
in Lemma~\ref{lemma:GenEquivRel} that, when $m_i=1,0\leq i\leq k$, the equivalence relation $\sim^{k,(1,1,\ldots,1)}_{\I}$ relation is the same as the usual equivalence relation of the projective space 
when the ideal $\mcl{I}$ can be expressed as a finite product of ideals whose radical are all distinct maximal ideals.
\begin{lemma}
\label{lemma:GenEquivRel}
Let $\R$ be a commutative ring with unity. Let $l\in \mbb{N}$ and $\I=\mcl{Q}_1\mcl{Q}_2\ldots\mcl{Q}_l$ where $\mcl{Q}_i\subs \R, 1\leq i\leq l$ is an ideal such that rad$(\mcl{Q}_i)=\mcl{M}_i$
is a maximal ideal and $\mcl{M}_i\neq \mcl{M}_j$ for $1\leq i\neq j\leq l$. 
Let \equ{\mcl{GCD}_{k+1}(\R)=\{(a_0,a_1,a_2,\ldots,a_k)\in \R^{k+1}\mid \us{i=0}{\os{k}{\sum}}(a_i)=\R\}.}  
Let \equ{(a_0,a_1,a_2,\ldots,a_k),(b_0,b_1,b_2,\ldots,b_k) \in \mcl{GCD}_{k+1}(\R).} 
Consider the following equivalence relation on $\mcl{GCD}_{k+1}(\R)$. We say \equ{(a_0,a_1,a_2,\ldots,a_k)\sim^R_{\I}(b_0,b_1,b_2,\ldots,b_k)} if for every $0\leq i,j \leq k$
we have $a_ib_j-b_ia_j\in \mcl{I}$. Then the two equivalence relations 
\equ{\sim^R_{\I},\sim^{\{k,(1,1,\ldots,1)\}}_{\I}} are identical.
\end{lemma}
\begin{proof}
Suppose $(a_0,a_1,a_2,\ldots,a_k)\sim^{\{k,(1,1,\ldots,1)\}}_{\I}(b_0,b_1,b_2,\ldots,b_k)$ then we clearly have \linebreak $(a_0,a_1,a_2,\ldots,a_k)\sim^R_{\I}(b_0,b_1,b_2,\ldots,b_k)$. 
Conversely if $(a_0,a_1,a_2,\ldots,a_k)\sim^R_{\I}(b_0,b_1,b_2,\ldots,b_k)$ then we observe the following. For every $1\leq r\leq l$ there exists $a_{i_r}\nin \mcl{M}_r$ for some $1\leq i_r\leq l$
because we have $\us{i=0}{\os{k}{\sum}}(a_i)=\R$. Now we have $b_{i_r}\nin \mcl{M}_r$. Otherwise $b_ja_{i_r}=b_ja_{i_r}-b_{i_r}a_j+b_{i_r}a_j \in \mcl{M}_r \Ra b_j\in \mcl{M}_r$ for all $0\leq j\leq k$.
This is a contradiction to $\us{i=0}{\os{k}{\sum}}(b_i)=\R$. So both $a_{i_r},b_{i_r}$ are units modulo $\mcl{Q}_r$. Hence we can choose $\gl_r\in \R$ such that $\gl_r\equiv \frac{b_{i_r}}{a_{i_r}} \mod \mcl{Q}_r$. 
We have $b_i\equiv \gl_r a_i\mod \mcl{Q}_r,0\leq i\leq k,1\leq r\leq l$.
By Chinese Remainder Theorem for ideals there exists $\gl\in \R$ such that $\gl \equiv \gl_r\equiv \frac{b_{i_r}}{a_{i_r}} \mod \mcl{Q}_r$ and hence $\ol{\gl}\in \bigg(\frac{\R}{\I}\bigg)^{*}$.
We also have $b_i\equiv \gl a_i \mod \I, 0\leq i\leq k$. Hence $(a_0,a_1,a_2,\ldots,a_k)\sim^{\{k,(1,1,\ldots,1)\}}_{\I}(b_0,b_1,b_2,\ldots,b_k)$. This proves the lemma.
\end{proof}

%%%%%%%%%%%%%%%%%%%%%%%%%%%%%%%%%%%%%%%%%%%%%%%%%%%%%%%%%%%%%%%%%%%%%%%%%%%%%%%%%%%%%%%%%
%%%%%%%%%%%%%%%%%%%%%%%%%%%%%%%%%%%%%%%%%%%%%%%%%%%%%%%%%%%%%%%%%%%%%%%%%%%%%%%%%%%%%%%%%
%%%%%%%%%%%%%%%%%%%%%%%%%%%%%%%%%%%%%%%%%%%%%%%%%%%%%%%%%%%%%%%%%%%%%%%%%%%%%%%%%%%%%%%%%

\section{\bf{Preliminaries}}
In this section we present some lemmas and propositions which are useful in the proof of main results.
\subsection{\bf{On arithmetic progressions}}
\label{sec:FundLemma}
In this section we prove a very useful lemma on arithmetic
progressions for integers and a proposition in the context of commutative rings with identity.
Remark~\ref{remark:FundLemma} below summarizes Lemma~\ref{lemma:FundLemma} and Proposition~\ref{prop:FundLemmaRings} in this section. 

%%%%%%%%%%%%%%%%%%%%%%%%%%%%%%%%%%%%%%%%%%%%%%%%%%%%%%%%%%%%%%%%%%%%%%%%%%%%%%%%%%%%%%%%%
%%%%%%%%%%%%%%%%%%%%%%%%%%%%%%%%%%%%%%%%%%%%%%%%%%%%%%%%%%%%%%%%%%%%%%%%%%%%%%%%%%%%%%%%%
%%%%%%%%%%%%%%%%%%%%%%%%%%%%%%%%%%%%%%%%%%%%%%%%%%%%%%%%%%%%%%%%%%%%%%%%%%%%%%%%%%%%%%%%%

\begin{lemma}[A lemma on Arithmetic Progressions for Integers]
\label{lemma:FundLemma}
~\\
Let $a,b\in \Z$ be integers such that $gcd(a,b)=1$. Let $m\in \Z$ be any non-zero integer. Then there exists $n_0\in \Z$ such that $gcd(a+n_0b,m)=1$.
\end{lemma}
\begin{proof}
Assume $a,b$ are both non-zero. Otherwise 
Lemma~\ref{lemma:FundLemma} is trivial. Let
$q_1,q_2,q_3,\ldots,q_t$ be the distinct prime factors of $m$.
Suppose $q \mid gcd(m,b)$ then $q \nmid a+nb$ for all $n \in
\mbb{Z}$. Such prime factors $q$ need not be considered. Let $q \mid
m, q \nmid b$. Then there exists $t_q \in \mbb{Z}$ such that the
exact set of elements in the given arithmetic progression divisible
by $q$ is given by \equ{\ldots, a+(t_q-2q)b, a+(t_q-q)b, a+t_qb,
a+(t_q+q)b,a+(t_q+2q)b \ldots} Since there are finitely many such
prime factors for $m$ which do not divide $b$ we get a set of
congruence conditions for the multiples of $b$ as $n \equiv t_q \mod\ q$. In order to get an $n_0$ we solve a different set of
congruence conditions for each such prime factor say for example $n
\equiv t_q+1 \mod\ q$. By Chinese remainder theorem we have such
solutions $n_0$ for $n$ which therefore satisfy $gcd(a+n_0b,m)=1$.
\end{proof}

%%%%%%%%%%%%%%%%%%%%%%%%%%%%%%%%%%%%%%%%%%%%%%%%%%%%%%%%%%%%%%%%%%%%%%%%%%%%%%%%%%%%%%%%%
%%%%%%%%%%%%%%%%%%%%%%%%%%%%%%%%%%%%%%%%%%%%%%%%%%%%%%%%%%%%%%%%%%%%%%%%%%%%%%%%%%%%%%%%%
%%%%%%%%%%%%%%%%%%%%%%%%%%%%%%%%%%%%%%%%%%%%%%%%%%%%%%%%%%%%%%%%%%%%%%%%%%%%%%%%%%%%%%%%%

\begin{prop}
\label{prop:FundLemmaRings} Let $\R$ be a commutative ring with identity. Let $f,g\in \R$ and they generate the unit ideal, that is, $(f)+(g)=\R$. Let $E$ be any
finite set maximal ideals in $\R$. Then there exists an element $a
\in \R$ such that $f+ag$ is a non-zero element in $\frac{\R}{\mcl{M}}$
for every $\mcl{M}\in E$.
\end{prop}
\begin{proof}
Let $E=\{\mcl{M}_1,\mcl{M}_2,\ldots,\mcl{M}_t\}$. If $g\in \mcl{M}_i$ then for all $a\in
\R,f+ag\nin \mcl{M}_i$. Otherwise both $f,g \in \mcl{M}_i$ which is a
contradiction to $1 \in (f,g)$.

Consider the finitely many maximal ideals $\mcl{M}\in E$ such
that $g \nin \mcl{M}$. Then there exists $t_{\mcl{M}}$ such that the
set \equ{\{t \mid f+tg \in \mcl{M}\}=t_{\mcl{M}}+\mcl{M}} a complete
arithmetic progression. This can be proved as follows. Since 
$g \nin \mcl{M}$ we have $(g)+\mcl{M}=\R$. So
there exists $t_{\mcl{M}}$ such that $f+t_{\mcl{M}}g \in \mcl{M}$.
If $f+tg \in \mcl{M}$ then $(t-t_{\mcl{M}})g \in \mcl{M}$. Hence
$t \in t_{\mcl{M}}+\mcl{M}$.

Since there are finitely many maximal ideals $\mcl{M}$ such that $g
\nin \mcl{M}$ in the set $E$ we get a finite set of congruence
conditions for the multiples $a$ of $g$ as $a \equiv t_{\mcl{M}}
\mod\ \mcl{M}$. In order to get an $a_0$ we solve a different set of
congruence conditions for each such maximal ideal in $E$ say for
example $a \equiv t_{\mcl{M}}+1 \mod\ \mcl{M}$. By Chinese Remainder
Theorem we have such solutions $a_0$ for $a$ which therefore satisfy
$f+a_0g \nin \mcl{M}$ for all maximal ideals $\mcl{M} \in E$ and
hence  $f+a_0g\nin \mcl{M}$ for every $\mcl{M} \in E$. This proves 
Proposition~\ref{prop:FundLemmaRings}.
\end{proof}

%%%%%%%%%%%%%%%%%%%%%%%%%%%%%%%%%%%%%%%%%%%%%%%%%%%%%%%%%%%%%%%%%%%%%%%%%%%%%%%%%%%%%%%%%
%%%%%%%%%%%%%%%%%%%%%%%%%%%%%%%%%%%%%%%%%%%%%%%%%%%%%%%%%%%%%%%%%%%%%%%%%%%%%%%%%%%%%%%%%
%%%%%%%%%%%%%%%%%%%%%%%%%%%%%%%%%%%%%%%%%%%%%%%%%%%%%%%%%%%%%%%%%%%%%%%%%%%%%%%%%%%%%%%%%

\begin{remark}
\label{remark:FundLemma}
If $a,b\in \Z,gcd(a,b)=1$ then there exist $x,y\in \Z$ such that $ax+by=1$.
Here we note that in general $x$ need not be one unless $a\equiv 1\mod\ b$.
However for any non-zero integer $m$ we can always choose $x=1$ to find
an integer $a+by$ such that $gcd(a+by,m)=1$. In the context of rings this observation
gives rise to elements which are outside a given finite set of maximal ideals.
\end{remark}

%%%%%%%%%%%%%%%%%%%%%%%%%%%%%%%%%%%%%%%%%%%%%%%%%%%%%%%%%%%%%%%%%%%%%%%%%%%%%%%%%%%%%%%%%
%%%%%%%%%%%%%%%%%%%%%%%%%%%%%%%%%%%%%%%%%%%%%%%%%%%%%%%%%%%%%%%%%%%%%%%%%%%%%%%%%%%%%%%%%
%%%%%%%%%%%%%%%%%%%%%%%%%%%%%%%%%%%%%%%%%%%%%%%%%%%%%%%%%%%%%%%%%%%%%%%%%%%%%%%%%%%%%%%%%

\subsection{\bf{The unital lemma}}
\label{sec:UnitalLemma}
In this section we prove unital Lemma~\ref{lemma:Unital} which is useful to obtain a
unit modulo a certain type of an ideal in a $k\operatorname{-}$row unital vector via an $SL_k(\R)\operatorname{-}$elementary
transformation. We define in Definition~\ref{defn:unitalset} below, when a finite subset of a commutative ring $\R$ is a unital set.

%%%%%%%%%%%%%%%%%%%%%%%%%%%%%%%%%%%%%%%%%%%%%%%%%%%%%%%%%%%%%%%%%%%%%%%%%%%%%%%%%%%%%%%%%
%%%%%%%%%%%%%%%%%%%%%%%%%%%%%%%%%%%%%%%%%%%%%%%%%%%%%%%%%%%%%%%%%%%%%%%%%%%%%%%%%%%%%%%%%
%%%%%%%%%%%%%%%%%%%%%%%%%%%%%%%%%%%%%%%%%%%%%%%%%%%%%%%%%%%%%%%%%%%%%%%%%%%%%%%%%%%%%%%%%

\begin{defn}
\label{defn:unitalset}
Let $\R$ be a commutative ring with unity. Let $k\in \mbb{N}$. We say a finite subset 
\equ{\{a_1,a_2,\ldots,a_k\}\subs \R} consisting of $k\operatorname{-}$elements (possibly with repetition) 
is unital or a unital set if the ideal generated by the elements of the set is a unit ideal.
\end{defn}

%%%%%%%%%%%%%%%%%%%%%%%%%%%%%%%%%%%%%%%%%%%%%%%%%%%%%%%%%%%%%%%%%%%%%%%%%%%%%%%%%%%%%%%%%
%%%%%%%%%%%%%%%%%%%%%%%%%%%%%%%%%%%%%%%%%%%%%%%%%%%%%%%%%%%%%%%%%%%%%%%%%%%%%%%%%%%%%%%%%
%%%%%%%%%%%%%%%%%%%%%%%%%%%%%%%%%%%%%%%%%%%%%%%%%%%%%%%%%%%%%%%%%%%%%%%%%%%%%%%%%%%%%%%%%

Based on the previous definition, we make a relevant definition, the unital set condition for an ideal.
\begin{defn}[Unital set condition for an ideal]
\label{defn:UnitalSetCond} Let $\R$ be a commutative ring with unity. Let $k\in \mbb{N}$ and $\mcl{I} \sbnq \mcl{\R}$ be an ideal. 
We say $\mcl{I}$ satisfies unital set condition $USC$ if for every unital set
$\{a_1,a_2,\ldots,a_k\} \subs \R$ with $k \geq 2$, there exists an
element $b \in (a_2,\ldots,a_k)$ such that $a_1+b$ is a unit modulo
$\mcl{I}$.
\end{defn}
\begin{example}
In the ring $\Z$ any ideal $0\neq \mcl{I}\subsetneq \Z$ satisfies USC using Lemma~\ref{lemma:FundLemma}. 
\end{example}
%%%%%%%%%%%%%%%%%%%%%%%%%%%%%%%%%%%%%%%%%%%%%%%%%%%%%%%%%%%%%%%%%%%%%%%%%%%%%%%%%%%%%%%%%
%%%%%%%%%%%%%%%%%%%%%%%%%%%%%%%%%%%%%%%%%%%%%%%%%%%%%%%%%%%%%%%%%%%%%%%%%%%%%%%%%%%%%%%%%
%%%%%%%%%%%%%%%%%%%%%%%%%%%%%%%%%%%%%%%%%%%%%%%%%%%%%%%%%%%%%%%%%%%%%%%%%%%%%%%%%%%%%%%%%
We state a proposition which gives a criterion for $USC$.
\begin{prop}
\label{prop:Unital} Let $\R$ be a commutative ring with unity. Let
$\mcl{J}\sbnq \R$ be an ideal contained in only a finitely many maximal
ideals. Then $\mcl{J}$ satisfies $USC$, that is, if $k \geq 2$ is a positive integer and if $\{a_1,a_2,\ldots,a_k\} \subs
\R$ is a unital set i.e. $\us{i=1}{\os{k}{\sum}}(a_i)=\mcl{R}$, then there exists $a \in (a_2,\ldots,a_k)$ such that
$a_1+a$ is a unit mod $\mcl{J}$.
\end{prop}
\begin{proof}
Let $\{\mcl{M}_i:1\leq i\leq t\}$ be the finite set of maximal
ideals containing $\mcl{J}$. For example $\mcl{J}$ could be a
product of maximal ideals. Since the set $\{a_1,a_2,\ldots,a_k\}$ is
unital there exists $d \in (a_2,a_3,\ldots,a_k)$ such that
$(a_1)+(d)=(1)$. We apply Proposition~\ref{prop:FundLemmaRings}, where $E=\{\mcl{M}_i:1\leq i\leq t\}$ to conclude that
there exists $n_0 \in \R$ such that $a=n_0d$ and $a_1+a=a_1+n_0d \nin
\mcl{M}_i$ for $1 \leq i\leq t$. This proves 
Proposition~\ref{prop:Unital}.
\end{proof}
\begin{example}
Let $\R$ be a prinicipal ideal domain or more generally a Dedekind domain. Then any non-zero ideal satisfies USC.
\end{example}
%%%%%%%%%%%%%%%%%%%%%%%%%%%%%%%%%%%%%%%%%%%%%%%%%%%%%%%%%%%%%%%%%%%%%%%%%%%%%%%%%%%%%%%%%
%%%%%%%%%%%%%%%%%%%%%%%%%%%%%%%%%%%%%%%%%%%%%%%%%%%%%%%%%%%%%%%%%%%%%%%%%%%%%%%%%%%%%%%%%
%%%%%%%%%%%%%%%%%%%%%%%%%%%%%%%%%%%%%%%%%%%%%%%%%%%%%%%%%%%%%%%%%%%%%%%%%%%%%%%%%%%%%%%%%

\begin{lemma}
\label{lemma:Unital} Let $\R$ be a commutative ring with unity and
$k \geq 2$ be a positive integer. Let $\{a_1,a_2,\ldots,a_k\} \subs
\R$ be a unital set i.e. $\us{i=1}{\os{k}{\sum}}(a_i)=\mcl{R}$ and
$E$ be a finite set of maximal ideals in $\R$. Then there exists
$a \in (a_2,\ldots,a_k)$ such that $a_1+a \nin
\mcl{M}$ for all $\mcl{M} \in E$.
\end{lemma}
\begin{proof}
The proof is essentially similar to Proposition~\ref{prop:Unital} even though we need not have to construct an ideal
$\mcl{J}$ which is contained in exactly the maximal ideals in the set $E$.
\end{proof}

%%%%%%%%%%%%%%%%%%%%%%%%%%%%%%%%%%%%%%%%%%%%%%%%%%%%%%%%%%%%%%%%%%%%%%%%%%%%%%%%%%%%%%%%%
%%%%%%%%%%%%%%%%%%%%%%%%%%%%%%%%%%%%%%%%%%%%%%%%%%%%%%%%%%%%%%%%%%%%%%%%%%%%%%%%%%%%%%%%%
%%%%%%%%%%%%%%%%%%%%%%%%%%%%%%%%%%%%%%%%%%%%%%%%%%%%%%%%%%%%%%%%%%%%%%%%%%%%%%%%%%%%%%%%%

\subsection{\bf{A result of strong approximation type}}
\label{sec:SMR}
~\\
This is a result of strong approximation type.
Here we give a criterion called the $USC$ which is
given in Definition~\ref{defn:UnitalSetCond} and mention the following
surjectivity theorem which is stated as:
\begin{theorem}
\label{theorem:SurjModIdeal} Let $\R$ be a commutative ring with
unity. Let $k\in \mbb{N}$. Let \equ{SL_k(\R)=\{A\in M_{k\times k}(\R) \mid \Det(A)=1\}}
Let $\mcl{I} \sbnq \R$ be an ideal which satisfies the unital
set condition (see Definition~\ref{defn:UnitalSetCond}). Then the reduction map
\equ{SL_k(\R) \lra SL_k(\frac{\R}{\mcl{I}})} is surjective. 
\end{theorem}
A proof of Theorem~\ref{theorem:SurjModIdeal} can be found in~(refer to Theorem $1.7$ on Page $338$) C.~P.~Anil~Kumar~\cite{CPAK}. A survey of results on strong approximation can be found in A.~S.~Rapinchuk~\cite{SA}. 

%%%%%%%%%%%%%%%%%%%%%%%%%%%%%%%%%%%%%%%%%%%%%%%%%%%%%%%%%%%%%%%%%%%%%%%%%%%%%%%%%%%%%%%%%
%%%%%%%%%%%%%%%%%%%%%%%%%%%%%%%%%%%%%%%%%%%%%%%%%%%%%%%%%%%%%%%%%%%%%%%%%%%%%%%%%%%%%%%%%
%%%%%%%%%%%%%%%%%%%%%%%%%%%%%%%%%%%%%%%%%%%%%%%%%%%%%%%%%%%%%%%%%%%%%%%%%%%%%%%%%%%%%%%%%

\subsection{\bf{Choice multiplier hypothesis for a tuple with respect to an ideal}}
We define a key concept known as choice multiplier hypothesis $CMH$ for a tuple with respect to an ideal. This definition is useful in the proof of main Theorem~\ref{theorem:GenCRTSURJ}.
\begin{defn}[$CMH$ for a tuple with respect to an Ideal]
\label{defn:CMH}
~\\
Let $\R$ be a commutative ring with unity and $\mcl{I}\sbnq \R$ be
any ideal. Let $n > 1$ be any positive integer and
$(x_1,x_2,\ldots,x_n) \in \R^n$ be such that $(x_1)+(x_2)+\ldots+(x_n)+\mcl{I}=\R$. Suppose $\R$ has the property
that there exist $a_1,a_2,\ldots,a_n \in \R$ such that
$(a_1)+(a_2)+\ldots+(a_n)=\R$ and $a_1x_1+a_2x_2+\ldots+a_nx_n \in
1+\mcl{I}$. Then we say $\R$ satisfies $CMH$
for the tuple $(x_1,x_2,\ldots,x_n) \in \R^n$ with respect to the ideal $\mcl{I}$. 
\end{defn}
\begin{remark}
A class of rings which satisfy CMH with respect to any proper ideal is given in Theorem~\ref{theorem:CMH}.
\end{remark}
We prove two lemmas which are useful in the proof of main Theorem~\ref{theorem:GenCRTSURJ}.

%%%%%%%%%%%%%%%%%%%%%%%%%%%%%%%%%%%%%%%%%%%%%%%%%%%%%%%%%%%%%%%%%%%%%%%%%%%%%%%%%%%%%%%%%
%%%%%%%%%%%%%%%%%%%%%%%%%%%%%%%%%%%%%%%%%%%%%%%%%%%%%%%%%%%%%%%%%%%%%%%%%%%%%%%%%%%%%%%%%
%%%%%%%%%%%%%%%%%%%%%%%%%%%%%%%%%%%%%%%%%%%%%%%%%%%%%%%%%%%%%%%%%%%%%%%%%%%%%%%%%%%%%%%%%

\begin{lemma}
\label{lemma:CMH}
Let $\R$ be a commutative ring with unity. Let $\mcl{I} \sbnq \R$ be an ideal.
$\R$ always satisfies $CMH$ with respect to the ideal $\mcl{I}$ for any positive integer $n>1$ for all tuples $(x_1,x_2,\ldots,x_n)\in \R^n$
when one of the $x_i:1 \leq i \leq n$ is a unit $\mod \mcl{I}$.
\end{lemma}
\begin{proof}
Let $n\geq 2$. Let $(x_1,x_2,\ldots,x_n)\in \R^n$. Without loss of generality let $x_1$ be a unit modulo $\mcl{I}$. Let $ax_1+t=1$ for some $a_1\in \R,t\in \I$. 
Then we choose $a_1=a,a_2=t$ and $a_3=\ldots=a_n=0$. We have $(a_1)+(a_2)=\R$ and $a_1x_1+a_2x_2=ax+tx_2=1+t(x_2-1)\in 1+\I$. This proves the lemma.
\end{proof}
\begin{lemma}
\label{lemma:CMHimpliesUnitalVect}
Let $\R$ be a commutative ring with unity and $\I\sbnq R$ be an ideal. Let $n\in \mbb{N}_{>1}$ and
$(x_1,x_2,\ldots,x_n) \in \R^n$ be such that $(x_1)+(x_2)+\ldots+(x_n)+\mcl{I}=\R$. Suppose $\R$ satisfies CMH for the tuple $(x_1,x_2,\ldots,x_n)\in \R^n$
with respect to the ideal $\I$. Then there exists $t_1,t_2,\ldots,t_n\in\I$ such that $\us{i=1}{\os{n}{\sum}}(x_i+t_i)=\R$. 
\end{lemma}
\begin{proof}
By $CMH$ let $a_1,a_2,\ldots,a_n \in \R$ such that $(a_1)+(a_2)+\ldots+(a_n)=\R$ and $a_1x_1+a_2x_2+\ldots+a_nx_n = 1-t \in 1+\mcl{I}$ where $t\in \I$.
Suppose $b_1a_1+b_2a_2+\ldots+b_na_n=1$. Then we have $a_1(x_1+tb_1)+a_2(x_2+tb_2)+\ldots+a_n(x_n+tb_n)=1$. Choosing $t_i=tb_i$ the lemma follows.
\end{proof}
\begin{theorem}
\label{theorem:CMH}
Let $\R$ be a Dedekind domain and $\mcl{I}\subsetneq \R$ be an ideal. Let $n\in \N_{>1}$ and $(x_1,x_2,\ldots,x_n) \in \R^n$ be such that $(x_1)+(x_2)+\ldots+(x_n)+\mcl{I}=\R$. Then $\R$ satisfies $CMH$ for the tuple $(x_1,x_2,\ldots,x_n) \in \R^n$ with respect to the ideal $\mcl{I}$.
\end{theorem}
\begin{proof}
Using Proposition $2.27$ in C.~P.~Anil Kumar~\cite{CPAK}, we obtain $t_1,t_2,\ldots,t_n\in \mcl{I}$ such that the set $\{x_1+t_1,x_2+t_2,\ldots,x_n+t_n\}$ is unital in $\R$. So there exists $a_1,a_2,\ldots,a_n\in \R$ such that $a_1(x_1+t_1)+a_2(x_2+t_2)+\ldots+a_n(x_n+t_n)=1$. Hence $(a_1)+(a_2)+\ldots+(a_n)=\R$ and $a_1x_1+a_2x_2+\ldots+a_nx_n\in 1+\mcl{I}$. 
\end{proof}
%%%%%%%%%%%%%%%%%%%%%%%%%%%%%%%%%%%%%%%%%%%%%%%%%%%%%%%%%%%%%%%%%%%%%%%%%%%%%%%%%%%%%%%%%
%%%%%%%%%%%%%%%%%%%%%%%%%%%%%%%%%%%%%%%%%%%%%%%%%%%%%%%%%%%%%%%%%%%%%%%%%%%%%%%%%%%%%%%%%
%%%%%%%%%%%%%%%%%%%%%%%%%%%%%%%%%%%%%%%%%%%%%%%%%%%%%%%%%%%%%%%%%%%%%%%%%%%%%%%%%%%%%%%%%

\section{\bf{Chinese Remainder Theorem for Generalized Projective Spaces}}
In this section we prove the first main result which concerns the surjectivity of the Chinese remainder 
reduction map associated to a projective space of an ideal with a given co-maximal ideal factorization by proving first main Theorem~\ref{theorem:GenCRTSURJ}.

%%%%%%%%%%%%%%%%%%%%%%%%%%%%%%%%%%%%%%%%%%%%%%%%%%%%%%%%%%%%%%%%%%%%%%%%%%%%%%%%%%%%%%%%%
%%%%%%%%%%%%%%%%%%%%%%%%%%%%%%%%%%%%%%%%%%%%%%%%%%%%%%%%%%%%%%%%%%%%%%%%%%%%%%%%%%%%%%%%%
%%%%%%%%%%%%%%%%%%%%%%%%%%%%%%%%%%%%%%%%%%%%%%%%%%%%%%%%%%%%%%%%%%%%%%%%%%%%%%%%%%%%%%%%%

\begin{proof}
If $\mcl{I}=\R$ then the proof is easy. Now we ignore unit ideals which occur in the factorization $\mcl{I}=\us{i=1}{\os{k}{\prod}}\mcl{I}_i$.
The theorem holds for $k=1$ and any $l\in \mbb{N}$ as the proof is immediate. We prove by induction on $k$. Let
\equa{([a_{10}:a_{11}:\ldots:a_{1l}],\ldots,&[a_{k0}:a_{k1}:\ldots:a_{kl}])
\in\\& \mbb{PF}^{l,(m_0,m_1,\ldots,m_l)}_{\mcl{I}_1} \times \mbb{PF}^{l,(m_0,m_1,\ldots,m_l)}_{\mcl{I}_2} \times
\ldots \times \mbb{PF}^{l,(m_0,m_1,\ldots,m_l)}_{\mcl{I}_k}.} By induction we
have an element $[b_0:b_1:b_2:\ldots:b_l] \in
\mbb{PF}^{l,(m_0,m_1,\ldots,m_l)}_{\mcl{I}_2\mcl{I}_3\ldots\mcl{I}_k}$ representing the
last $k-1$ elements. Consider the matrix \equ{A=\begin{pmatrix}
a_{10}  & a_{11}  & \cdots & a_{1,l-1} & a_{1l}\\
b_0& b_1& \cdots & b_{l-1} & b_l
\end{pmatrix}}
where we have $\us{i=0}{\os{l}{\sum}}(a_{1i})=\R=\us{i=0}{\os{l}{\sum}}(b_i)$.
We change the matrix $A$ to a suitable matrix $B$ by applying elementary column operations on $A$ or right multiply $A$ by matrices in $SL_{l+1}(\R)$. We keep track of the matrices in $SL_{l+1}(\R)$ used for mulitplication for back tracking later by their inverses at the end. 
Now the ideal $\mcl{I}_1$ is contained in only a finitely many maximal ideals. By using Proposition~\ref{prop:Unital} and a suitable application of $SL_{l+1}(\R)$ matrix, we can assume $a_{10}$ is a unit modulo $\mcl{I}_1$. 
By finding inverse of this element modulo $\mcl{I}_1$ and hence again by a suitable application of $SL_{l+1}(\R)$ matrix, the matrix $A$ can
be transformed to the following matrix $B$, where $a_0$ in the first row is a unit modulo $\mcl{I}_1$ and $a_i\in \mcl{I}_1$ for $1\leq i\leq l$ and also $\us{i=0}{\os{l}{\sum}}(a_i)=\R$.
\equ{B=\begin{pmatrix}
a_0  & a_1  & \cdots & a_{l-1} & a_l\\
c_0& c_1& \cdots & c_{l-1} & c_l
\end{pmatrix}}

The ideal $\mcl{I}_2\ldots \mcl{I}_k$ is contained in only a finitely many maximal ideals. If $c_0$ is not a unit $\mod \mcl{I}_2\ldots \mcl{I}_k$ then, by using
Proposition~\ref{prop:Unital} and a suitable application of $SL_{l+1}(\R)$ matrix, we can assume the first element $c_0$ in the second row of $B$ is a unit modulo $\mcl{I}_2\ldots\mcl{I}_k$
and the first element $a_0$ in the first row will still remain a unit modulo $\mcl{I}_1$ after the transformation.
We have the following facts on the matrix $B$ now.

\begin{enumerate}
\item $a_0$ is a unit modulo $\mcl{I}_1$. 
\item $a_i \in \mcl{I}_1$ for $1\leq i\leq l$.
\item $\us{i=0}{\os{l}{\sum}}(a_i)=\R$.
\item $c_0$ is a unit modulo $\mcl{I}_2\ldots\mcl{I}_k$.
\item $\us{i=0}{\os{l}{\sum}}(c_i)=\R$.
\end{enumerate}

By usual Chinese Remainder Theorem let $x_0\equiv a_0\mod \mcl{I}_1,x_0\equiv c_0 \mod \mcl{I}_2\ldots\mcl{I}_k$. For $1\leq i\leq l$ let $x_i\in \mcl{I}_1,x_i\equiv c_i\mod \mcl{I}_2\ldots\mcl{I}_k$.
Then we have $(x_0)+(x_1)+\ldots+(x_l)+\mcl{I}_1=\R,(x_0)+(x_1)+\ldots+(x_l)+\mcl{I}_2\ldots\mcl{I}_k=\R$ which implies $(x_0)+(x_1)+\ldots+(x_l)+\mcl{I}=\R$. Moreover $x_0$ is a unit modulo $\I$ as it is a unit modulo
both $\mcl{I}_1$ and $\mcl{I}_2\ldots\mcl{I}_k$. Since $l\geq 1$ by using Lemmas~\ref{lemma:CMH},~\ref{lemma:CMHimpliesUnitalVect} (as $CMH$ is satisfied), there exist $t_0,t_1,\ldots,t_l \in \I$
such that $\us{i=0}{\os{l}{\sum}} (x_i+t_i)=\R$ and a required element for $B$ is given by $[x_0+t_0:x_1+t_1:\ldots:x_l+t_l]\in \mbb{PF}^{l,(m_0,m_1,\ldots,m_l)}_{\mcl{I}}$.  By back tracking we can get a required element for $A$. 
Hence the induction step is completed thereby proving surjectivity.

Now we prove injectivity of the map which is slightly easier. Suppose $[a_0:a_1:\ldots:a_l],[b_0:b_1:\ldots:b_l]$ have the same image. Then there exist $\gl_i\in \R,1\leq i\leq k$ such that $\ol{\gl_i}\in \frac{\R}{\mcl{I}_i}$ is a unit and we have $a_j=\gl_i^{m_j}b_j \mod \mcl{I}_i, 0\leq j\leq l, 1\leq i\leq k$. Now we use usual chinese remainder theorem to obtain an element $\gl\in \R$ such that $\gl\equiv \gl_i\mod \mcl{I}_i$. Hence it follows that $\ol{\gl}\in \frac{\R}{\mcl{I}}$ is a unit and $a_j=\gl^{m_j}b_j \mod \mcl{I}, 0\leq j\leq l$. So we get $[a_0:a_1:\ldots:a_l]=[b_0:b_1:\ldots:b_l]$. 
Now Theorem~\ref{theorem:GenCRTSURJ} follows.
\end{proof}

%%%%%%%%%%%%%%%%%%%%%%%%%%%%%%%%%%%%%%%%%%%%%%%%%%%%%%%%%%%%%%%%%%%%%%%%%%%%%%%%%%%%%%%%%
%%%%%%%%%%%%%%%%%%%%%%%%%%%%%%%%%%%%%%%%%%%%%%%%%%%%%%%%%%%%%%%%%%%%%%%%%%%%%%%%%%%%%%%%%
%%%%%%%%%%%%%%%%%%%%%%%%%%%%%%%%%%%%%%%%%%%%%%%%%%%%%%%%%%%%%%%%%%%%%%%%%%%%%%%%%%%%%%%%%

\section{\bf{On the Generalized Surjectivity Theorem for Generalized Projective Spaces}}
In this section we prove two surjectivity theorems in the context of generalized projective spaces. More precisely we prove
following Theorem~\ref{theorem:GenSurjMainGenIdeals} and main Theorem~\ref{theorem:FullGenSurj}.
But first we mention a remark.
\begin{remark}
	In a Dedekind domain for example, any non-Jacobson element is contained in only a finitely many maximal ideals. 
\end{remark}

Theorem~\ref{theorem:GenSurjMainGenIdeals} is stated as follows where a condition on the ring is required:
\begin{theorem}
\label{theorem:GenSurjMainGenIdeals}
Let $\R$ be a commutative ring with unity. Suppose every non-Jacobson element is contained in only a finitely many maximal ideals.
Let $k\in \mbb{N}$ and $\mcl{I}_0,\mcl{I}_1,\ldots,\mcl{I}_k$ be $(k+1)$ co-maximal ideals in $\R$. 
Also if there is exactly one proper ideal $\mcl{I}_j$ for some $0\leq j\leq k$ then we suppose it is contained in only a finitely many maximal ideals.
Let $A_{(k+1)\times (k+1)}=[a_{i,j}]_{0\leq i,j \leq k}\in M_{(k+1)\times (k+1)}(\R)$ such that for every $0\leq i\leq k$ the $i^{th}\operatorname{-}$row is unital, that is, $\us{j=0}{\os{k}{\sum}} (a_{i,j})=\R$
for $0\leq i\leq k$. Then there exists $B=[a_{i,j}]_{0\leq i,j \leq k}\in SL_{k+1}(\R)$ such that we have 
$a_{i,j}\equiv b_{i,j}\mod \mcl{I}_i,0\leq i,j\leq k$.
\end{theorem}

%%%%%%%%%%%%%%%%%%%%%%%%%%%%%%%%%%%%%%%%%%%%%%%%%%%%%%%%%%%%%%%%%%%%%%%%%%%%%%%%%%%%%%%%%
%%%%%%%%%%%%%%%%%%%%%%%%%%%%%%%%%%%%%%%%%%%%%%%%%%%%%%%%%%%%%%%%%%%%%%%%%%%%%%%%%%%%%%%%%
%%%%%%%%%%%%%%%%%%%%%%%%%%%%%%%%%%%%%%%%%%%%%%%%%%%%%%%%%%%%%%%%%%%%%%%%%%%%%%%%%%%%%%%%%

Before proving these two theorems we need some results which we will state and prove. 
\begin{lemma}[A lemma on ideal avoidance]
\label{lemma:IdealAvoid}
Let $\R$ be a commutative ring with unity and $\mcl{I}\subs \R$ be an ideal.
Let $r\in \mbb{N}$ and $\mcl{N}_1,\mcl{N}_2,\ldots,\mcl{N}_r$ be maximal ideals
in $\R$ such that $\mcl{I} \nsbq \mcl{N}_i, 1\leq i\leq r$. Then we have 
\equ{\mcl{I}\bs \bigg(\us{i=1}{\os{r}{\bigcup}}\mcl{N}_i\bigg)=\mcl{I}\bs \bigg(\us{i=1}{\os{r}{\bigcup}}\mcl{I}\mcl{N}_i\bigg)\neq \es.}
\end{lemma}
\begin{proof}
This follows from ideal avoidance as we have $\mcl{I}\nsbq \mcl{N}_i, 1\leq i\leq r$. Refer to Proposition $1.11$ in Chapter $1$ on Page $8$ of M.~F.~Atiyah and I.~G.~MacDonald~\cite{AM}.
\end{proof}

%%%%%%%%%%%%%%%%%%%%%%%%%%%%%%%%%%%%%%%%%%%%%%%%%%%%%%%%%%%%%%%%%%%%%%%%%%%%%%%%%%%%%%%%%
%%%%%%%%%%%%%%%%%%%%%%%%%%%%%%%%%%%%%%%%%%%%%%%%%%%%%%%%%%%%%%%%%%%%%%%%%%%%%%%%%%%%%%%%%
%%%%%%%%%%%%%%%%%%%%%%%%%%%%%%%%%%%%%%%%%%%%%%%%%%%%%%%%%%%%%%%%%%%%%%%%%%%%%%%%%%%%%%%%%

\begin{prop}
\label{prop:ChoiceMaxElements}
Let $\R$ be a commutative ring with unity. Suppose every non-Jacobson element is contained in only a finitely many maximal ideals in $\R$. Let $k\in \mbb{N}$ and $\mcl{I}_1,\mcl{I}_2,\ldots,\mcl{I}_k$ be co-maximal ideals in $\R$.
Then there exist elements $q_i\in \mcl{I}_i, 1\leq i\leq k$ such that $(q_i)+(q_j)=\R, 1\leq i\neq j \leq k$. 
\end{prop}
\begin{proof}
If $k=1$ the conclusion is vacuously true and for $k=2$ the conclusion holds because of co-maximality. If $\mcl{I}_i=\R$ for some $1\leq i\leq k$ then we can choose for that value $i, q_i=1\in \R=\mcl{I}_i$ trivially.
So we assume without loss of generality that none of the ideals is a unit ideal and $k>2$. Hence by mutual co-maximality, since $k>2$, there exists a non-Jacobson element in each of the ideals $\mcl{I}_i,1\leq i\leq k$.
So each $\mcl{I}_i, 1\leq i\leq k$ is contained in only a finitely many maximal ideals. Let $\mcl{I}_i$ be contained in the maximal ideals $\mcl{M}_{i1},\mcl{M}_{i2},\ldots,\mcl{M}_{ir_i}$ with $0<r_i\in \mbb{N}$.
By using Lemma~\ref{lemma:IdealAvoid} let \equ{q_1\in \mcl{I}_1\bs \bigg(\us{i=2}{\os{k}{\bigcup}}\us{j=1}{\os{r_i}{\bigcup}}\mcl{M}_{ij}\bigg) \neq \es.} The element $q_1$ is non-Jacobson. Hence there 
exists a finitely many maximal ideals $\mcl{N}_1,\mcl{N}_2,\ldots,\mcl{N}_r$ containing $q_1$ by hypothesis. We observe that $\mcl{I}_2\nsbq \mcl{N}_i,1\leq i\leq r$.
So again using Lemma~\ref{lemma:IdealAvoid} let \equ{q_2\in \mcl{I}_2\bs \bigg(\us{j=1}{\os{r}{\bigcup}}\mcl{N}_j\bigg) \bs \bigg(\us{i=1,i\neq 2}{\os{k}{\bigcup}}\us{j=1}{\os{r_i}{\bigcup}}\mcl{M}_{ij}\bigg) \neq \es.} 
The element $q_2$ is non-Jacobson. There exists a finitely many maximal ideals  $\mcl{P}_1,\mcl{P}_2,\ldots,\mcl{P}_s$ containing $q_2$ by hypothesis. We observe that 
$\mcl{I}_3\nsbq \mcl{N}_i,1\leq i\leq r$ and $\mcl{I}_3 \nsbq \mcl{P}_i,1\leq i\leq s$. Again using Lemma~\ref{lemma:IdealAvoid} let 
\equ{q_3\in \mcl{I}_3\bs \bigg(\us{l=1}{\os{s}{\bigcup}}\mcl{P}_l\bigg) \bs \bigg(\us{j=1}{\os{r}{\bigcup}}\mcl{N}_j\bigg) \bs \bigg(\us{i=1,i\neq 3}{\os{k}{\bigcup}}\us{j=1}{\os{r_i}{\bigcup}}\mcl{M}_{ij}\bigg) \neq \es.}
The element $q_3$ is non-Jacobson. Continuing this procedure we obtain elements $q_i\in \mcl{I}_i, 1\leq i\leq k$. 
Since there are no common maximal ideals containing $q_i,q_j$ for $1\leq i\neq j \leq k$ we have $(q_i)+(q_j)=\R$. This proves the proposition.
\end{proof}

%%%%%%%%%%%%%%%%%%%%%%%%%%%%%%%%%%%%%%%%%%%%%%%%%%%%%%%%%%%%%%%%%%%%%%%%%%%%%%%%%%%%%%%%%
%%%%%%%%%%%%%%%%%%%%%%%%%%%%%%%%%%%%%%%%%%%%%%%%%%%%%%%%%%%%%%%%%%%%%%%%%%%%%%%%%%%%%%%%%
%%%%%%%%%%%%%%%%%%%%%%%%%%%%%%%%%%%%%%%%%%%%%%%%%%%%%%%%%%%%%%%%%%%%%%%%%%%%%%%%%%%%%%%%%

\begin{prop}
\label{prop:diagdetone}
Let $\R$ be a commutative ring with unity. Suppose every non-Jacobson element is contained in only a finitely many maximal ideals in $\R$. Let $1<k\in \mbb{N}$ and $\mcl{I}_1,\mcl{I}_2,\ldots,\mcl{I}_k$ be 
mutually co-maximal ideals in $\R$. Let $a_1,a_2,\ldots,a_k\in \R$ be such that for $1\leq i\leq k, a_i$ is a unit modulo $\mcl{I}_i$ if $\mcl{I}_i\neq \R$. 
Then there exist $d_i \equiv a_i \mod \mcl{I}_i, 1\leq i\leq k$ such that 
\equ{d_1d_2\ldots d_k \equiv 1 \mod \mcl{I}_1\mcl{I}_2\ldots \mcl{I}_k.} 
\end{prop}
\begin{proof}
If all $\mcl{I}_i$ are unit ideals then we choose $d_i=1,1\leq i\leq k$. If one of them (say) $\mcl{I}_1\neq\R$ and $\mcl{I}_2=\mcl{I}_3=\ldots =\mcl{I}_k=\R$ then, let $z_1\in \R$ be such that $z_1a_1\equiv 1\mod \mcl{I}_1$.
Choose $d_1=a_1,d_2=z_1,d_3=\ldots =d_k=1$. Now we can assume that there are at least two ideals (say) $\mcl{I}_1\neq \R\neq \mcl{I}_2$. Here we choose $d_j=1$ if $\mcl{I}_j=\R$ for some $j>2$.
Hence we ignore unit ideals and assume that none of the ideals are unit ideals and $k\geq 2$. 

We have $a_i, 1\leq i\leq k$ are all non-Jacobson as each $a_i$ is a unit modulo a proper ideal $\mcl{I}_i$ which is
contained in at least one maximal ideal. So for any $1\leq i\leq k$, if $b_i \equiv a_i\mod \mcl{I}_i$ then $b_i$ is unit modulo $\mcl{I}_i$. Hence $b_i$ is non-Jacobson. 
Using Proposition~\ref{prop:ChoiceMaxElements} there exist $q_i\in \mcl{I}_i$ such that $(q_i)+(q_j)=\R, 1\leq i\neq j\leq k$. We have $q_i,1\leq i\leq k$ are non-Jacobson as each $q_i$
is a unit modulo $\mcl{I}_j$ for any $j\neq i, 1\leq j \leq k$. Since $(a_i)+\mcl{I}_i=\R$, let $\ti{q}_i\in \mcl{I}_i$ be such that 
$(a_i)+(\ti{q}_i)=\R$. Now using Proposition~\ref{prop:FundLemmaRings} let $\ti{d}_1=a_1+\ga_1\ti{q}_1\nin \mcl{N}$ for all maximal ideals $\mcl{N}$ containing any $q_i, 1\leq i\leq k$. 
Again using Proposition~\ref{prop:FundLemmaRings} we choose $\ti{d}_2=a_2+\ga_2\ti{q}_2 \nin \mcl{N}$ for all maximal ideals containing $\ti{d}_1$ or any $q_i, 1\leq i\leq k$. 
Continuing this procedure we obtain elements $\ti{d}_i\in a_i+\mcl{I}_i$ such that 
\begin{itemize}
\item There is no common maximal ideal containing $\ti{d}_i,\ti{d}_j$ for $1\leq i\neq j\leq k$.
\item There is no common maximal ideal containing $\ti{d}_i,q_j$ for $1\leq i,j\leq k$.
\item Also we have by choice there is no common maximal ideal containing $q_i,q_j$ for $1\leq i\neq j\leq k$.
\end{itemize}   

Let \equ{e_1=\frac{\us{j=1}{\os{k}{\prod}} \ti{d}_j}{\ti{d}_1}\frac{\us{j=1}{\os{k}{\prod}} q_j}{q_k}} and for $2\leq i\leq k$ let 
\equ{e_i=\frac{\us{j=1}{\os{k}{\prod}} \ti{d}_j}{\ti{d}_i}\frac{\us{j=1}{\os{k}{\prod}} q_j}{q_{i-1}}.} Then there is no maximal ideal containing all the elements $e_1,e_2,\ldots,e_k$.
Hence we have \equ{\us{i=1}{\os{k}{\sum}} (e_i)=\R.}

Let $x_1,x_2,\ldots,x_k,X_1,X_2,\ldots,X_k$ be variables such that  
\equ{x_1=\frac{\us{j=1}{\os{k}{\prod}} q_j}{q_kq_1}X_1,x_i=\frac{\us{j=1}{\os{k}{\prod}} q_j}{q_{i-1}q_i}X_i, 2\leq i\leq k.}
Consider the equation 
\equa{\us{i=1}{\os{k}{\prod}}\big(\ti{d}_i+q_ix_i\big)&=1 \Ra\\
\us{i=1}{\os{k}{\sum}}e_iX_i=1- \us{i=1}{\os{k}{\prod}}\ti{d}_i + f[X_1,X_2,\ldots,X_n]&\text{ for some }f \in \bigg(\us{i=1}{\os{k}{\prod}}\mcl{I}_i\bigg)\R[X_1,X_2,\ldots,X_n].}
We choose values $X_i=\gb_i\in \R$ (which exist) such that 
\equ{\us{i=1}{\os{k}{\sum}}e_i\gb_i=1- \us{i=1}{\os{k}{\prod}}\ti{d}_i.}
Then we obtain the product
\equ{\us{i=1}{\os{k}{\prod}}\big(\ti{d}_i+q_i\gga_i\big)\equiv 1 \mod \bigg(\us{i=1}{\os{k}{\prod}}\mcl{I}_i\bigg)} where 
\equ{\gga_1=\frac{\us{j=1}{\os{k}{\prod}} q_j}{q_kq_1}\gb_1,\gga_i=\frac{\us{j=1}{\os{k}{\prod}} q_j}{q_{i-1}q_i}\gb_i, 2\leq i\leq k.}
Choosing $d_i=\ti{d}_i+q_i\gga_i\in a_i+\mcl{I}_i$ we have 
\equ{d_1d_2\ldots d_k\equiv 1 \mod \mcl{I}_1\mcl{I}_2\ldots\mcl{I}_k} and the proposition follows. 
\end{proof}

%%%%%%%%%%%%%%%%%%%%%%%%%%%%%%%%%%%%%%%%%%%%%%%%%%%%%%%%%%%%%%%%%%%%%%%%%%%%%%%%%%%%%%%%%
%%%%%%%%%%%%%%%%%%%%%%%%%%%%%%%%%%%%%%%%%%%%%%%%%%%%%%%%%%%%%%%%%%%%%%%%%%%%%%%%%%%%%%%%%
%%%%%%%%%%%%%%%%%%%%%%%%%%%%%%%%%%%%%%%%%%%%%%%%%%%%%%%%%%%%%%%%%%%%%%%%%%%%%%%%%%%%%%%%%

In the following important and very useful Theorem~\ref{theorem:GenSurjMain} we consider ideals whose radicals are distinct maximal ideals, instead of just co-maximal ideals, 
each of which is contained in only a finitely many maximal ideals. Hence we use the notation $\mcl{Q}_i,0\leq i\leq k$ instead of notation $\mcl{I}_i,0\leq i\leq k$.
This theorem and its proof motivates the statement of Theorem~\ref{theorem:GenSurjMainGenIdeals} and its proof. Theorem~\ref{theorem:GenSurjMain} is stated as follows
where a condition on the ring is required:
\begin{theorem}
\label{theorem:GenSurjMain}
Let $\R$ be a commutative ring with unity. Suppose every non-Jacobson element is contained in only a finitely many maximal ideals.
Let $k\in \mbb{N}$ and $\mcl{Q}_0,\mcl{Q}_1,\ldots,\mcl{Q}_k$ be $(k+1)$ co-maximal ideals in $\R$ whose radicals rad$(\mcl{Q}_i)=\mcl{M}_i$ are distinct maximal ideals. 
Let $A_{(k+1)\times (k+1)}=[a_{i,j}]_{0\leq i,j \leq k}\in M_{(k+1)\times (k+1)}(\R)$ such that for every $0\leq i\leq k$ the $i^{th}\operatorname{-}$row is unital, that is, $\us{j=0}{\os{k}{\sum}} (a_{i,j})=\R$
for $0\leq i\leq k$. Then there exists $B=[a_{i,j}]_{0\leq i,j \leq k}\in SL_{k+1}(\R)$ such that we have 
$a_{i,j}\equiv b_{i,j}\mod \mcl{Q}_i,0\leq i,j\leq k$.
\end{theorem}
\begin{proof}
First we make the following observations.
\begin{enumerate}
\item The $i^{th}\operatorname{-}$row of $A_{(k+1)\times (k+1)}$ is unital if and only if the $i^{th}\operatorname{-}$row of $A_{(k+1)\times (k+1)}.C$ is for any 
$C\in SL_{k+1}(\R)$.
\item The conclusion holds for the matrix $A_{(k+1)\times (k+1)}$ if and only if the conclusion holds for the matrix $A_{(k+1)\times (k+1)}.C$ for any $C\in SL_{k+1}(\R)$.
\item For any $0\leq i\leq k$, we can replace the $i^{th}\operatorname{-}$row $[a_{i,0}\ \ldots \ a_{i,k}]$ of $A_{(k+1)\times (k+1)}$ by another unital row $[\ti{a}_{i,0}\ \ldots \ \ti{a}_{i,k}]$ 
such that $a_{i,j}\equiv \ti{a}_{i,j}\mod \mcl{Q}_i, 0\leq j\leq k$.  
\end{enumerate}
We prove this theorem in several steps with the central idea being to transform $A_{(k+1)\times (k+1)}$ to another matrix for which the conclusion of the theorem holds.
\begin{enumerate}[label=Step(\Alph*):]
\item By applying Proposition~\ref{prop:ChoiceMaxElements} we have that there exist $q_i\in \mcl{Q}_i,0\leq i\leq k$ such that $(q_i)+(q_j)=\R$ for $0\leq i\neq j\leq k$.
\item There exists $0\leq j\leq k$ such that $a_{0,j}\nin rad(\mcl{Q}_0)=\mcl{M}_0$. If $a_{0,0}\in \mcl{M}_0$ and $a_{0,j_0}\nin \mcl{M}_0$ for some $j_0\neq 0$ then we add the $j_0^{th}\operatorname{-}$column to the 
$0^{th}\operatorname{-}$column to obtain an element, also denoted by $a_{0,0}$, such that $a_{0,0}\nin \mcl{M}_0$ and hence a unit modulo $\mcl{Q}_0$. Let $z_{0,0}\in \R$ such that $z_{0,0}a_{0,0} \equiv 1 \mod \mcl{Q}_0$.
Now we add $-z_{0,0}a_{0,j}$ times the $0^{th}\operatorname{-}$column to the $j^{th}\operatorname{-}$column for $1\leq j\leq k$. The $0^{th}\operatorname{-}$row becomes $[a_{0,0}\ \ldots \ a_{0,k}]$ with the following
properties.
\begin{itemize}
\item $a_{0,0} \nin \mcl{M}_0$ and hence a unit modulo $\mcl{Q}_0$.
\item $a_{0,j}\in \mcl{Q}_0$ for $1\leq j\leq k$.
\item $\us{j=0}{\os{k}{\sum}}(a_{0,j})=\R$, that is the $0^{th}\operatorname{-}$row is a unital vector.
\end{itemize}
In the subsequent steps we preserve these properties of the $0^{th}\operatorname{-}$row.
\item We inductively consider the $i^{th}\operatorname{-}$row of $A_{(k+1)\times (k+1)}$.
Here if $a_{i,i}\in \mcl{M}_i=$rad$(\mcl{Q}_i)$ then there exists $0\leq j_i\leq k,j_i\neq i$ such that $a_{i,j_i}\nin \mcl{M}_i$ and hence a unit modulo 
$\mcl{Q}_i$. If $0\leq j_i<i$ then we add $q_0q_1q_2\ldots q_{i-1}$ times $j_i^{th}\operatorname{-}$column to $i^{th}\operatorname{-}$column. We note that 
\begin{itemize}
\item $q_1q_2\ldots q_{i-1}a_{i,j_i}\nin \mcl{M}_i$.
\item $q_0q_1q_2\ldots q_{i-1}a_{l,j_i}\in \mcl{Q}_l$ for any $0\leq l\leq i-1$.
\end{itemize}
If $i<j_i\leq k$ then we just add the $j_i^{th}\operatorname{-}$column to the $i^{th}\operatorname{-}$column.
Hence after adding the column we obtain an element, also denoted by $a_{i,i}$, the diagonal entry in the $i^{th}\operatorname{-}$row such that $a_{i,i}\nin \mcl{M}_0$ and hence a unit modulo $\mcl{Q}_i$.
Let $z_{i,i}\in \R$ be such that $z_{i,i}a_{i,i} \equiv 1 \mod \mcl{Q}_i$. Add $-z_{i,i}a_{i,j}$ times $i^{th}\operatorname{-}$column to the $j^{th}\operatorname{-}$column for $0\leq j\leq k, j\neq i$.
Now we have the following properties for the matrix $A_{(k+1)\times (k+1)}$.
\begin{itemize}
\item $a_{l,l}\nin \mcl{M}_l$ for $0\leq l\leq i$.
\item $a_{l,j}\in \mcl{Q}_j$ for $j\neq l, 0\leq j\leq k, 0\leq l\leq i$
\item All rows of $A_{(k+1)\times (k+1)}$ are unital. 
\end{itemize}
\item We continue this procedure till the last $k^{th}\operatorname{-}$row.
After this procedure we have the following properties for $A_{(k+1)\times (k+1)}$.
\begin{itemize}
\item The diagonal entry $a_{i,i}\nin \mcl{M}_i$ and hence a unit modulo $\mcl{Q}_i$ for $0\leq i\leq k$.
\item The non-diagonal entry $a_{i,j}\in \mcl{Q}_i$ for $0\leq i\neq j\leq k$.
\item All rows of $A_{(k+1)\times (k+1)}$ are unital. 
\end{itemize}
\item Consider only the diagonal part of the matrix $A_{(k+1)\times (k+1)}$, that is, the matrix $D_{(k+1)\times (k+1)}=diag(a_{0,0},a_{1,1},\ldots,a_{k,k})$. 
We use Proposition~\ref{prop:diagdetone} to change the diagonal matrix to $D_{(k+1)\times (k+1)}$ to $\ti{D}_{(k+1)\times (k+1)}=diag(d_0,d_1,\ldots,d_k)$ such that we have 
$d_i\equiv a_{i,i}\mod \mcl{Q}_i,0\leq i\leq k$ and $d_0d_1\ldots d_k\equiv 1\mod \mcl{Q}_0\mcl{Q}_1\ldots\mcl{Q}_k$. Hence 
\equ{\ol{\ti{D}}_{(k+1)\times (k+1)} \in SL_{k+1}\bigg(\frac{\R}{\mcl{Q}_0\mcl{Q}_1\ldots \mcl{Q}_k}\bigg).}
We observe by using Proposition~\ref{prop:Unital}, that, $\mcl{Q}_0\mcl{Q}_1\ldots \mcl{Q}_k$ satisfies 
unital set condition as it is contained in only a finitely many maximal ideals. Hence by an application of Theorem~\ref{theorem:SurjModIdeal} we conclude that the reduction map 
\equ{SL_{k+1}(\R) \lra SL_{k+1}\bigg(\frac{\R}{\mcl{Q}_0\mcl{Q}_1\ldots \mcl{Q}_k}\bigg)} is surjective. Therefore there exists a matrix $B=[b_{i,j}]_{0\leq i,j\leq k}\in SL_{k+1}(\R)$ such that we have 
\equ{\ol{B}=\ol{\ti{D}}_{(k+1)\times (k+1)} \in SL_{k+1}\bigg(\frac{\R}{\mcl{Q}_0\mcl{Q}_1\ldots \mcl{Q}_k}\bigg).}
\item This matrix $B$ is a required matrix. We observe the following. 
\begin{itemize}
\item The diagonal entries $b_{i,i}\equiv d_i\equiv a_{i,i} \mod \mcl{Q}_i$.
\item The non-diagonal entries $b_{i,j}\in \mcl{Q}_i$ and hence $b_{i,j}\equiv a_{i,j}$ for $0\leq i\neq j\leq k$.
\end{itemize}
\end{enumerate}
We have completed the proof of the theorem in Steps $(A)-(F)$.
\end{proof}
We generalize Theorem~\ref{theorem:GenSurjMain} and prove Theorem~\ref{theorem:GenSurjMainGenIdeals} for co-maximal ideals $\mcl{I}_i,0\leq i\leq k$ each of which is contained in only a finitely many maximal ideals instead 
of ideals $\mcl{Q}_i,0\leq i\leq k$ whose radicals are distinct maximal ideals. But first we state a useful proposition.

%%%%%%%%%%%%%%%%%%%%%%%%%%%%%%%%%%%%%%%%%%%%%%%%%%%%%%%%%%%%%%%%%%%%%%%%%%%%%%%%%%%%%%%%%
%%%%%%%%%%%%%%%%%%%%%%%%%%%%%%%%%%%%%%%%%%%%%%%%%%%%%%%%%%%%%%%%%%%%%%%%%%%%%%%%%%%%%%%%%
%%%%%%%%%%%%%%%%%%%%%%%%%%%%%%%%%%%%%%%%%%%%%%%%%%%%%%%%%%%%%%%%%%%%%%%%%%%%%%%%%%%%%%%%%

\begin{prop}
\label{prop:bringunit}
Let $\R$ be a commutative ring with unity. Let $k\in \mbb{N}\cup\{0\}$ and $(a_0,a_1,\ldots,a_k)\linebreak \in \R^{k+1}$ be a unital vector that is $\us{i=0}{\os{k}{\sum}}(a_i)=\R$. 
Let $\mcl{I} \sbnq \R, \mcl{J} \subs \R$ be pairwise co-maximal ideals and $\mcl{I}$ is contained in only a  finitely many ideals. For any subscript $0\leq i\leq k$ there exist 
\equ{x_0,x_1,\ldots, x_{i-1}\in \mcl{J} \text{ and }x_{i+1},\ldots,x_k\in \R} such that the element 
\equ{\us{j=0}{\os{i-1}{\sum}}x_ja_j+a_i+\us{j=i+1}{\os{k}{\sum}}x_ja_j}
is a unit modulo the ideal $\mcl{I}$. 
\end{prop}
\begin{proof}
Let $q_1\in \mcl{I},q_2\in \mcl{J}$ be such that $q_1+q_2=1$. Let $\mcl{I}$ be contained in distinct maximal ideals $\mcl{M}_1,\mcl{M}_2,\ldots,\mcl{M}_n$.
By renumbering if necessary let 
\begin{itemize}
\item $a_i\nin \mcl{M}_1\cup\ldots\cup\mcl{M}_{r_1}$,
\item $a_i\in \mcl{M}_{r_1+1}\cap \ldots \cap \mcl{M}_n$.
\end{itemize}
Let $l_0=i,r_0=0$. There exists a subscript $l_1$ such that $a_{l_1}\nin \mcl{M}_{r_1+1}$. Again by renumbering the subscripts $r_1+2,r_1+3,\ldots,n$ we assume that 
\begin{itemize}
\item $a_{l_1}\nin \mcl{M}_{r_1+1}\cup\ldots\cup\mcl{M}_{r_2}$,
\item $a_{l_1}\in \mcl{M}_{r_2+1}\cap \ldots \cap \mcl{M}_n$.
\end{itemize}
We continue this procedure finitely many times to obtain distinct subscripts $l_0,l_1,l_2,\ldots,l_t$  and subscripts $0=r_0<1\leq r_1<r_2<\ldots<r_{t+1}=n$ with the property that 
\begin{itemize}
\item $a_{l_{j-1}} \nin \mcl{M}_{r_{j-1}+1} \cup \ldots \cup \mcl{M}_{r_j}$ for $1\leq j\leq t$,
\item $a_{l_{j-1}} \in \mcl{M}_{r_j+1}\cap \ldots \cap \mcl{M}_n$ for $1\leq j\leq t$.
\item $a_{l_t} \nin \mcl{M}_{r_t+1}\cup \ldots \cup \mcl{M}_{r_{t+1}=n}$.
\end{itemize}
For $1\leq j\leq t$ let \equ{y_{l_j}\in \bigg(\us{i=1}{\os{r_j}{\bigcap}}\mcl{M}_i\bigg)\bigcap \bigg(\us{i=r_{j+1}+1}{\os{n}{\bigcap}}\mcl{M}_i\bigg)\bs \bigg(\us{i=r_j+1}{\os{r_{j+1}}{\bigcup}}\mcl{M}_i\bigg) \neq \es.}
Consider the element 
\equ{a=\us{j=1,l_j<l_0=i}{\os{t}{\sum}}q_2y_{l_j}a_{l_j}+a_i+\us{j=1,l_j>l_0=i}{\os{t}{\sum}}y_{l_j}a_{l_j}.}
This element $a$ does not belong to any of the maximal ideals $\mcl{M}_1,\ldots,\mcl{M}_n$. Hence it is unit modulo $\mcl{I}$. Now we take 
\begin{itemize}
\item $x_{l_j} = q_2y_{l_j}\in \mcl{J}$ if $l_j<l_0=i, 1\leq j\leq t$.
\item $x_{l_j} = y_{l_j}$ if $l_j>l_0=i,1\leq j\leq t$.
\item $x_l=0$ if $l\nin \{l_0,l_1,\ldots,l_t\},0\leq l\leq k$. 
\end{itemize}
This completes the proof of the proposition.
\end{proof}

%%%%%%%%%%%%%%%%%%%%%%%%%%%%%%%%%%%%%%%%%%%%%%%%%%%%%%%%%%%%%%%%%%%%%%%%%%%%%%%%%%%%%%%%%
%%%%%%%%%%%%%%%%%%%%%%%%%%%%%%%%%%%%%%%%%%%%%%%%%%%%%%%%%%%%%%%%%%%%%%%%%%%%%%%%%%%%%%%%%
%%%%%%%%%%%%%%%%%%%%%%%%%%%%%%%%%%%%%%%%%%%%%%%%%%%%%%%%%%%%%%%%%%%%%%%%%%%%%%%%%%%%%%%%%

Now we prove Theorem~\ref{theorem:GenSurjMainGenIdeals}.
\begin{proof}
First we make the same observations that we made earlier in the proof of Theorem~\ref{theorem:GenSurjMain}.
\begin{enumerate}
\item The $i^{th}\operatorname{-}$row of $A_{(k+1)\times (k+1)}$ is unital if and only if the $i^{th}\operatorname{-}$row of $A_{(k+1)\times (k+1)}.C$ is for any 
$C\in SL_{k+1}(\R)$.
\item The conclusion holds for the matrix $A_{(k+1)\times (k+1)}$ if and only if the conclusion holds for the matrix $A_{(k+1)\times (k+1)}.C$ for any $C\in SL_{k+1}(\R)$.
\item For any $0\leq i\leq k$, we can replace the $i^{th}\operatorname{-}$row $[a_{i,0}\ \ldots \ a_{i,k}]$ of $A_{(k+1)\times (k+1)}$ by another unital row $[\ti{a}_{i,0}\ \ldots \ \ti{a}_{i,k}]$ 
such that $a_{i,j}\equiv \ti{a}_{i,j}\mod \mcl{I}_i, 0\leq j\leq k$.  
\end{enumerate}
If there are at least two mutually co-maximal proper ideals then all the proper ideals among them contain a non-Jacobson element because of co-maximality. 
So each ideal $\mcl{I}_j,0\leq j\leq k$ is contained in only a finitely many maximal ideals.
Suppose there exists only one ideal say $\mcl{I}_0 \neq \R$ and $\mcl{I}_1=\mcl{I}_2=\mcl{I}_3=\ldots=\mcl{I}_k=\R$ then by hypothesis $\mcl{I}_0$ is contained in only a finitely many maximal ideals.
Hence in all cases we can assume that each ideal $\mcl{I}_j,0\leq j\leq k$ is contained in only a finitely many maximal ideals. If $\us{i=0}{\os{k}{\prod}}\mcl{I}_i=\R$ then the proof is easy. So
we also assume that $\us{i=0}{\os{k}{\prod}}\mcl{I}_i\neq \R$.

We prove this theorem in several steps with the central idea being to transform $A_{(k+1)\times (k+1)}$ to another matrix for which the conclusion of the theorem holds.
Steps$(A)-(C)$ do not require the fact that every non-Jacobson element is contained in only a finitely many maximal ideals. In Step$(D)$ in the application of Proposition~\ref{prop:diagdetone} requires this fact.

\begin{enumerate}[label=Step(\Alph*):]
\item If $\mcl{I}_0\neq \R$ then using Proposition~\ref{prop:Unital} we can make $a_{0,0}$ a unit modulo $\mcl{I}_0$ by applying an $SL_{k+1}(\R)$ transformation as $\mcl{I}_0$ is contained in only a finitely many maximal ideals.
Let $z_{0,0}\in \R$ such that $z_{0,0}a_{0,0} \equiv 1 \mod \mcl{I}_0$.
Now we add $-z_{0,0}a_{0,j}$ times the $0^{th}\operatorname{-}$column to the $j^{th}\operatorname{-}$column for $1\leq j\leq k$. The $0^{th}\operatorname{-}$row becomes $[a_{0,0}\ \ldots \ a_{0,k}]$ with the following
properties.
\begin{itemize}
\item $a_{0,0}$ is a unit modulo $\mcl{I}_0$ if $\mcl{I}_0\neq \R$.
\item $a_{0,j}\in \mcl{I}_0$ for $1\leq j\leq k$.
\item $\us{j=0}{\os{k}{\sum}}(a_{0,j})=\R$, that is the $0^{th}\operatorname{-}$row is a unital vector.
\end{itemize}
In the subsequent steps we preserve these properties of the $0^{th}\operatorname{-}$row.
\item We inductively consider the $i^{th}\operatorname{-}$row of $A_{(k+1)\times (k+1)}$.
Here if $\mcl{I}_i\neq \R$ and $a_{i,i}$ is not a unit modulo $\mcl{I}_i$ then we use Proposition~\ref{prop:bringunit} for the subscript $i$ and for the ideals $\mcl{I}=\mcl{I}_i$ and
$\mcl{J}=\mcl{I}_0\mcl{I}_1\ldots\mcl{I}_{i-1}$ which is co-maximal with
$\mcl{I}_i$ to make $a_{i,i}$ a unit modulo $\mcl{I}_i$. In this procedure the matrix $A$ will have the following properties.
\begin{itemize}
\item $a_{l,l}$ is a unit modulo $\mcl{I}_l$ for $0\leq l\leq i$ if $\mcl{I}_l\neq \R$.
\item $a_{l,j}\in \mcl{I}_j$ for $j\neq l, 0\leq j\leq k, 0\leq l\leq i$ when Proposition~\ref{prop:bringunit} is applied approriately, (that is, for the $i^{th}\operatorname{-}$row the subscript
$i$ is chosen, the ideal $\mcl{I}=\mcl{I}_i$ is chosen and the ideal $\mcl{J}=\mcl{I}_0\mcl{I}_1\ldots\mcl{I}_{i-1}$ is chosen in Proposition~\ref{prop:bringunit}).
\item All rows of $A_{(k+1)\times (k+1)}$ are unital. 
\end{itemize}
\item We continue this procedure till the last $k^{th}\operatorname{-}$row.
After this procedure we have the following properties for $A_{(k+1)\times (k+1)}$.
\begin{itemize}
\item The diagonal entry $a_{i,i}$ is a unit modulo $\mcl{I}_i$ if $\mcl{I}_i\neq \R$ for $0\leq i\leq k$.
\item The non-diagonal entry $a_{i,j}\in \mcl{I}_i$ for $0\leq i\neq j\leq k$.
\item All rows of $A_{(k+1)\times (k+1)}$ are unital. 
\end{itemize}
\item Consider only the diagonal part of the matrix $A_{(k+1)\times (k+1)}$, that is, the matrix $D_{(k+1)\times (k+1)}=diag(a_{0,0},a_{1,1},\ldots,a_{k,k})$. 
We use Proposition~\ref{prop:diagdetone} to change the diagonal matrix to $D_{(k+1)\times (k+1)}$ to $\ti{D}_{(k+1)\times (k+1)}=diag(d_0,d_1,\ldots,d_k)$ such that we have 
$d_i\equiv a_{i,i}\mod \mcl{I}_i,0\leq i\leq k$ and $d_0d_1\ldots d_k\equiv 1\mod \mcl{I}_0\mcl{I}_1\ldots\mcl{I}_k$. The ideal $\mcl{I}_0\mcl{I}_1\ldots\mcl{I}_k \neq \R$ by assumption. Hence 
\equ{\ol{\ti{D}}_{(k+1)\times (k+1)} \in SL_{k+1}\bigg(\frac{\R}{\mcl{I}_0\mcl{I}_1\ldots \mcl{I}_k}\bigg).}
We observe by using Proposition~\ref{prop:Unital}, that, $\mcl{I}_0\mcl{I}_1\ldots \mcl{I}_k$ satisfies 
unital set condition as it is contained in only a finitely many maximal ideals. Hence by an application of Theorem~\ref{theorem:SurjModIdeal} we conclude that the reduction map 
\equ{SL_{k+1}(\R) \lra SL_{k+1}\bigg(\frac{\R}{\mcl{I}_0\mcl{I}_1\ldots \mcl{I}_k}\bigg)} is surjective. Therefore there exists a matrix $B=[b_{i,j}]_{0\leq i,j\leq k}\in SL_{k+1}(\R)$ such that we have 
\equ{\ol{B}=\ol{\ti{D}}_{(k+1)\times (k+1)} \in SL_{k+1}\bigg(\frac{\R}{\mcl{I}_0\mcl{I}_1\ldots \mcl{I}_k}\bigg).}
\item This matrix $B$ is a required matrix. We observe the following. 
\begin{itemize}
\item The diagonal entries $b_{i,i}\equiv d_i\equiv a_{i,i} \mod \mcl{I}_i$.
\item The non-diagonal entries $b_{i,j}\in \mcl{I}_i$ and hence $b_{i,j}\equiv a_{i,j}$ for $0\leq i\neq j\leq k$.
\end{itemize}
\end{enumerate}
We have completed the proof of the theorem in Steps $(A)-(E)$.
\end{proof}

%%%%%%%%%%%%%%%%%%%%%%%%%%%%%%%%%%%%%%%%%%%%%%%%%%%%%%%%%%%%%%%%%%%%%%%%%%%%%%%%%%%%%%%%%
%%%%%%%%%%%%%%%%%%%%%%%%%%%%%%%%%%%%%%%%%%%%%%%%%%%%%%%%%%%%%%%%%%%%%%%%%%%%%%%%%%%%%%%%%
%%%%%%%%%%%%%%%%%%%%%%%%%%%%%%%%%%%%%%%%%%%%%%%%%%%%%%%%%%%%%%%%%%%%%%%%%%%%%%%%%%%%%%%%%

Now we prove second main Theorem~\ref{theorem:FullGenSurj}.
\begin{proof}
Let \equ{\big([a_{0,0}:a_{0,1}:\ldots: a_{0,k}],[a_{1,0}:a_{1,1}:\ldots: a_{1,k}],\ldots,[a_{k,0}:a_{k,1}:\ldots: a_{k,k}]\big)\in \us{i=0}{\os{k}{\prod}}\mbb{PF}^{k,(m^i_0,m^i_1,\ldots,m^i_k)}_{\I_i}.}
Consider $A_{(k+1)\times (k+1)}=[a_{i,j}]_{0\leq i,j\leq k}\in M_{(k+1)\times (k+1)}(\R)$ for which Theorem~\ref{theorem:GenSurjMainGenIdeals} can be applied. Therefore we get 
$B=[b_{i,j}]_{0\leq i,j\leq k} \in SL_{k+1}(\R)$ such that $b_{i,j}\equiv a_{i,j} \mod \mcl{I}_i,0\leq i,j\leq k$. Hence we get 
\equ{[b_{i,0}:b_{i,1}:\ldots: b_{i,k}]=[a_{i,0}:a_{i,1}:\ldots: a_{i,k}]\in \mbb{PF}^{k,(m^i_0,m^i_1,\ldots,m^i_k)}_{\I_i},0\leq i\leq k.}
This proves second main Theorem~\ref{theorem:FullGenSurj}.
\end{proof}

%%%%%%%%%%%%%%%%%%%%%%%%%%%%%%%%%%%%%%%%%%%%%%%%%%%%%%%%%%%%%%%%%%%%%%%%%%%%%%%%%%%%%%%%%
%%%%%%%%%%%%%%%%%%%%%%%%%%%%%%%%%%%%%%%%%%%%%%%%%%%%%%%%%%%%%%%%%%%%%%%%%%%%%%%%%%%%%%%%%
%%%%%%%%%%%%%%%%%%%%%%%%%%%%%%%%%%%%%%%%%%%%%%%%%%%%%%%%%%%%%%%%%%%%%%%%%%%%%%%%%%%%%%%%%

\section{\bf{Generalized Projective Spaces: Revisited}}

In Theorems~\ref{theorem:GenSurjMainGenIdeals},~\ref{theorem:FullGenSurj} we have the hypothesis that every non-Jacobson element in the ring $\R$ is contained in only a finitely many maximal ideals. In this section we remove this condition on the ring with one more assumption on the generalized projective space and prove third main Theorem~\ref{theorem:FullGenSurjOne}. 

%%%%%%%%%%%%%%%%%%%%%%%%%%%%%%%%%%%%%%%%%%%%%%%%%%%%%%%%%%%%%%%%%%%%%%%%%%%%%%%%%%%%%%%%%
%%%%%%%%%%%%%%%%%%%%%%%%%%%%%%%%%%%%%%%%%%%%%%%%%%%%%%%%%%%%%%%%%%%%%%%%%%%%%%%%%%%%%%%%%
%%%%%%%%%%%%%%%%%%%%%%%%%%%%%%%%%%%%%%%%%%%%%%%%%%%%%%%%%%%%%%%%%%%%%%%%%%%%%%%%%%%%%%%%%

\begin{proof}
First assume $\gs$ is identity, that is, $\gs(i)=i,m^i_i=1,0\leq i\leq k$.
Let \equ{\big([a_{0,0}:a_{0,1}:\ldots: a_{0,k}],[a_{1,0}:a_{1,1}:\ldots: a_{1,k}],\ldots,[a_{k,0}:a_{k,1}:\ldots: a_{k,k}]\big)\in \us{i=0}{\os{k}{\prod}}\mbb{PF}^{k,(m^i_0,m^i_1,\ldots,m^i_k)}_{\I_i}.}
Consider its corresponding matrix $A_{(k+1)\times (k+1)}=[a_{i,j}]_{0\leq i,j\leq k}\in M_{(k+1)\times (k+1)}(\R)$.

%We make some observations.
%\begin{enumerate}
%\item The image of the map in the theorem is invariant under $SL_{k+1}(\R)$ action on the right.  
%\item The $i^{th}\operatorname{-}$row of $A_{(k+1)\times (k+1)}$ is unital if and only if the $i^{th}\operatorname{-}$row of $A_{(k+1)\times (k+1)}.C$ is for any $C\in SL_{k+1}(\R)$.
%\item The element corresponding to the matrix $A_{(k+1)\times (k+1)}$ is in the image if and only if the element corresponding to the matrix $A_{(k+1)\times (k+1)}.C$ for any $C\in SL_{k+1}(\R)$ is in the image.
%\item For any $0\leq i\leq k$, we can replace the $i^{th}\operatorname{-}$row $[a_{i,0}\ \ldots \ a_{i,k}]$ of $A_{(k+1)\times (k+1)}$ by another unital row $[\ti{a}_{i,0}\ \ldots \ \ti{a}_{i,k}]$ such that $a_{i,j}\equiv \ti{a}_{i,j}\mod \mcl{I}_i, 0\leq j\leq k$ which will not the change projective space element. 
%\item Also for any $0\leq i\leq k$, we can replace the $i^{th}\operatorname{-}$row $[a_{i,0}\ \ldots \ a_{i,k}]$ of $A_{(k+1)\times (k+1)}$ by another unital row $[\ti{a}_{i,0}\ \ldots \ \ti{a}_{i,k}]$  such that $[a_{i,0}:\ldots:a_{i,k}]=[\ti{a}_{i,0}:\ldots:\ti{a}_{i,k}]\in \mbb{PF}^{k,(m^i_0,m^i_1,\ldots,m^i_k)}_{\I_i}$.
%\end{enumerate}
We prove this theorem in several steps with the central idea being to transform $A_{(k+1)\times (k+1)}$ to another matrix whose corresponding element in the product of projective spaces is in the image.
We perform Steps $(A)-(C)$ in the proof of Theorem~\ref{theorem:GenSurjMainGenIdeals} here as well. To perform these three steps, we do not require the fact that every non-Jacobson element is contained in 
only a finitely many maximal ideals. Now we have the following properties for the matrix $A_{(k+1)\times (k+1)}$.
\begin{itemize}
\item The diagonal entry $a_{i,i}$ is a unit modulo $\mcl{I}_i$ if $\mcl{I}_i\neq \R$ for $0\leq i\leq k$.
\item The non-diagonal entry $a_{i,j}\in \mcl{I}_i$ for $0\leq i\neq j\leq k$.
\item All rows of $A_{(k+1)\times (k+1)}$ are unital. 
\end{itemize}
We use the fact that $m_i^i=1$ and replace the unital $i^{th}\operatorname{-}$row by $e_i=[0\ \ \ldots \ \ 0\ \ 1\ \ 0\ \ldots \ 0]$ both of which give rise to the same projective space element in 
$\mbb{PF}^{k,(m^i_0,m^i_1,\ldots,m^i_k)}_{\I_i}$. This is possible only because $m_i^i=1$. So we have reduced the matrix $A_{(k+1)\times (k+1)}$ to identity matrix whose corresponding element in the product
of projective spaces is in the image. This proves third main Theorem~\ref{theorem:FullGenSurjOne} when $\gs$ is identity. 

If $\gs$ is any even permutation then we permute the projective spaces $\mbb{PF}^{k,(m^i_0,m^i_1,\ldots,m^i_k)}_{\I_i}, 0\leq i\leq k$ so that we can assume $\gs$ is identity. Hence this case follows.

If $\gs$ is an odd permutation then we not only permute the projective spaces $\mbb{PF}^{k,(m^i_0,m^i_1,\ldots,m^i_k)}_{\I_i}, 0\leq i\leq k$ but we also use a flipping automorphism $F$ for one of the projective spaces
say \equ{F: \mbb{PF}^{k,(m^i_0,m^i_1,\ldots,m^i_k)}_{\I_0} \lra \mbb{PF}^{k,(m^i_0,m^i_1,\ldots,m^i_k)}_{\I_0}} given by 
\equ{F([a_{0,0}:a_{0,1}:\ldots: a_{0,k}])=[-a_{0,0}:-a_{0,1}:\ldots:-a_{0,k}]}
to flip the sign. With this we can assume that $\gs$ is identity. Hence this case also follows.

This completes the proof of third main Theorem~\ref{theorem:FullGenSurjOne} 
\end{proof}

%%%%%%%%%%%%%%%%%%%%%%%%%%%%%%%%%%%%%%%%%%%%%%%%%%%%%%%%%%%%%%%%%%%%%%%%%%%%%%%%%%%%%%%%%
%%%%%%%%%%%%%%%%%%%%%%%%%%%%%%%%%%%%%%%%%%%%%%%%%%%%%%%%%%%%%%%%%%%%%%%%%%%%%%%%%%%%%%%%%
%%%%%%%%%%%%%%%%%%%%%%%%%%%%%%%%%%%%%%%%%%%%%%%%%%%%%%%%%%%%%%%%%%%%%%%%%%%%%%%%%%%%%%%%%

\begin{remark}
Having proved Theorem~\ref{theorem:FullGenSurjOne}, we have generalized Theorem $1.8$ on Page $339$ of C.~P.~Anil~Kumar~\cite{CPAK} for ordinary projective spaces by removing the Dedekind type domain condition on the ring.
On the other hand we have similar type of conditions on the ring, though not exactly the same, in Theorem~\ref{theorem:FullGenSurj}, in the context of generalized projective spaces.
In the next section we prove a theorem in general for one dimensional projective spaces which removes the conditions on the ring and conditions on the values $m^i_j$..
\end{remark}

%%%%%%%%%%%%%%%%%%%%%%%%%%%%%%%%%%%%%%%%%%%%%%%%%%%%%%%%%%%%%%%%%%%%%%%%%%%%%%%%%%%%%%%%%
%%%%%%%%%%%%%%%%%%%%%%%%%%%%%%%%%%%%%%%%%%%%%%%%%%%%%%%%%%%%%%%%%%%%%%%%%%%%%%%%%%%%%%%%%
%%%%%%%%%%%%%%%%%%%%%%%%%%%%%%%%%%%%%%%%%%%%%%%%%%%%%%%%%%%%%%%%%%%%%%%%%%%%%%%%%%%%%%%%%

\section{\bf{Surjectivity of the map $SL_2(\R)\lra \mbb{PF}^{1,(m^1_0,m^1_1)}_{\I_1} \times \mbb{PF}^{1,(m^2_0,m^2_1)}_{\I_2}$ in General}}
In this section we prove surjectivity of a map in general without any conditions on the ring $\R$ in two dimensions. It only requires the fact that each of the pair of co-maximal ideals is contained in only a 
finitely many maximal ideals. The theorem is stated as follows.
\begin{theorem}
\label{theorem:FullGenSurjTwoDim}
Let $\R$ be a commutative ring with unity. Let $\mcl{I}_1,\mcl{I}_2$ be a pair of pairwise co-maximal ideals in $\R$ each of which is contained in only a finitely many maximal ideals and
$m_j^i\in \mbb{N}, i=1,2,j=0,1$.
Then the map 
\equ{SL_2(\R) \lra \mbb{PF}^{1,(m^1_0,m^1_1)}_{\I_1} \times \mbb{PF}^{1,(m^2_0,m^2_1)}_{\I_2}} given by
\equ{\mattwo xyzw \lra ([x:y],[z:w])}
is surjective.
\end{theorem}

%%%%%%%%%%%%%%%%%%%%%%%%%%%%%%%%%%%%%%%%%%%%%%%%%%%%%%%%%%%%%%%%%%%%%%%%%%%%%%%%%%%%%%%%%
%%%%%%%%%%%%%%%%%%%%%%%%%%%%%%%%%%%%%%%%%%%%%%%%%%%%%%%%%%%%%%%%%%%%%%%%%%%%%%%%%%%%%%%%%
%%%%%%%%%%%%%%%%%%%%%%%%%%%%%%%%%%%%%%%%%%%%%%%%%%%%%%%%%%%%%%%%%%%%%%%%%%%%%%%%%%%%%%%%%
 
We state and prove a theorem below which is required to prove Theorem~\ref{theorem:FullGenSurjTwoDim}. The theorem is quite general and is stated as follows:
\begin{theorem}
\label{theorem:FullGenSurjMainTwoDim}
Let $\R$ be a commutative ring with unity and $\mcl{I}_1,\mcl{I}_2\subs \R$ be a pair of co-maximal ideals. Let $A,B,C,D\in \R$ be such that 
\begin{itemize}
\item either $(A)+\mcl{I}_1=\R$ or $(B)+\mcl{I}_1=\R$.
\item either $(C)+\mcl{I}_2=\R$ or $(D)+\mcl{I}_2=\R$.
\end{itemize}
Then there exists a matrix \equ{\mattwo abcd\in SL_2(\R)} such that $a\equiv A \mod \mcl{I}_1, b \equiv B \mod \mcl{I}_1, c \equiv C \mod \mcl{I}_2, d \equiv D \mod \mcl{I}_2$.
\end{theorem}
\begin{proof}
~\\
\begin{enumerate}[label=Case(\arabic*):]
\item Consider the case when $(A)+\mcl{I}_1=\R$ and $(C)+\mcl{I}_2=\R$. Since $\mcl{I}_1+\mcl{I}_2=\R$, there exist $p_1\in \mcl{I}_1,p_2\in \mcl{I}_2$ such that $p_1+p_2=1$. Let $x,y\in \R$ be such that 
$A+px+C+qy=1$. Let $\ti{A}=A+p_1x,\ti{B}=C+p_2y$ so that $\ti{A}+\ti{C}=1$. Then we have $(\ti{A})+\mcl{I}_1=\R,(C)+\mcl{I}_2=\R,(\ti{A})+(\ti{C})=\R$. Hence we have \equ{(\ti{C})\mcl{I}_1+(\ti{A})\mcl{I}_2=\R.}
There exist $\ti{p}_1\in \mcl{I}_1,\ti{p}_2\in \mcl{I}_2$ such that $(\ti{A}\ti{p}_2)+(\ti{C}\ti{p}_1)=\R$. Hence there exists a matrix 
\equ{\mattwo {\ti{A}=A+p_1x}{B+\ti{p}_1u}{\ti{C}=C+p_2y}{D+\ti{p}_2v}\in SL_2(\R),}
that is, $\ti{A}\ti{p}_2v-\ti{C}\ti{p}_1u=1-\ti{A}D+B\ti{C}$ for some $u,v\in \R$. This proves the theorem in this case.
\item The case when $(B)+\mcl{I}_1=\R$ and $(D)+\mcl{I}_2=\R$ is similar to the previous case.
\item Consider the case $(A)+\mcl{I}_1=\R,(D)+\mcl{I}_2=\R$ and $(B)+\mcl{I}_1 \subsetneq \R,(C)+\mcl{I}_2\subsetneq \R$.
~\\
\begin{enumerate}[label=Step(\Roman*):]
\item We have $\mcl{I}_1+(A)=\R,\mcl{I}_1+\mcl{I}_2=\R \Ra \mcl{I}_1+(A)\mcl{I}_2=\R$. Let $p_1\in \mcl{I}_1,p_2\in (A)\mcl{I}_2$ such that $p_1+p_2=1$ where $p_2=A\ti{p}_2$ for some $\ti{p}_2\in \mcl{I}_2$.
Hence we have $(p_1^2)+(p_2)=\R$. There exist $r,s\in \R$ such that $p_1^2r+p_2s=1-D$. Let $D_1=D+p_2s$. Then $D_1+p_1^2r=1 \Ra (D_1)+(p_1^2)=\R$. 
\item $\R=(p_1^2)+(p_2)=(p_1^2)+(A\ti{p}_2)\Ra (p_1^2)+(A)=\R$. By Chinese Remainder Theorem, let $E\in \R$ such that $AE \equiv 1\mod (p_1^2)$ and $E \equiv 1\mod (D_1)$. Hence $(E)+(D_1)=\R,(E)+(p_1)=\R$. 
\item Since $(p_1^2)+(p_2)=R,(p_1^2)+(E)=\R$ we have $(p_1^2)+(p_2E)=\R$. Hence there exist $x,y\in \R$ such that $p_1^2y+p_2Ex=E-p_1B-D_1$. Let $\ti{B}=B+p_1y,\ti{D}=D_1+p_2Ex$. So $p_1\ti{B}+\ti{D}=E$. Then we have 
\equa{&p_1\ti{B}+\ti{D}=E \Ra E\in (\ti{B})+(\ti{D})\Ra D_1\in (\ti{B})+(\ti{D}) \Ra (\ti{B})+(\ti{D})=\R\\
&p_1\ti{B}+\ti{D}=E \Ra \ti{D}\equiv E\mod (p_1) \Ra A\ti{D} \equiv AE \equiv 1 \mod (p_1).}% \Ra (p_1)+(\ti{D})=\R.} 
\item $(D)+\mcl{I}_2=\R \Ra (D_1)+\mcl{I}_2=R\Ra (\ti{D})+\mcl{I}_2=\R$. There exist $p_2'\in \mcl{I}_2,w\in \R$ such that $w\ti{D}+p_2'=1$. Also we have $p_1+p_2=1$.
Hence $w\ti{D}p_1+p_2''=1$ where $p_2''=w\ti{D}p_2+p_2'p_1+p_2'p_2 \in \mcl{I}_2$. Moreover we have $(\ti{D})+(p_2'')=\R=(p_1)+(p_2'')$.
\item Since $(\ti{D})+(\ti{B})=\R,(\ti{D})+(p_2'')=\R$ we have $(\ti{D})+(\ti{B}p_2'')=\R$. So there exists a matrix 
\equ{\mattwo {a}{\ti{B}}{C+p_2''v}{\ti{D}}\in SL_2(\R),\text{ that is, } \ti{D}a-\ti{B}p_2''v=1+\ti{B}C \text{ for some }a,v\in \R.}
We conclude $\ti{D}a\equiv 1\mod (\ti{B})$ and for every $z\in \R$ we have  
\equan{zequation}{\ti{D}(a+\ti{B}p_2''z)-\ti{B}p_2''(v+\ti{D}z)=1+\ti{B}C.} 
\item We have $\ti{D}A\equiv 1\mod (p_1) \Ra \ti{D}A-1\in (p_1)\subs (p_1)+(\ti{B})$. Similarly $\ti{D}a \equiv 1\mod (\ti{B}) \Ra \ti{D}a-1\in (\ti{B})\subs (p_1)+(\ti{B})$
that is \equ{\ti{D}a \equiv 1 \equiv \ti{D}A \mod (p_1)+(\ti{B}) \Ra A \equiv a \mod (p_1)+(\ti{B})}
by canceling $\ti{D}$ because $\ti{D}$ is invertible modulo the ideal $(p_1)+(\ti{B})$.
\item Let $A-a=p_1r_1+\ti{B}r_2$ for some $r_1,r_2\in \R$. Then $A-a\equiv \ti{B}r_2\mod (p_1)$. Since $(p_2'')+(p_1)=\R$ we have $p_2''$ invertible modulo $(p_1)$. Hence there exists $r_2'\in \R$
such that $A-a\equiv \ti{B}p_2''r_2' \mod (p_1)$ where $r_2'\equiv r_1(p_2'')^{-1}\mod (p_1)$.
Hence there exist $z_0\in \R$ such that \equ{A\equiv a+\ti{B}p_2''z_0 \mod (p_1) \Ra A \equiv a+\ti{B}p_2''z_0 \mod \mcl{I}_1.}
\item Using this value $z_0$ for $z$ in Equation~\ref{Eq:zequation} we conclude that 
\equ{\mattwo {a+\ti{B}p_2''z_0}{\ti{B}}{C+p_2''(v+\ti{D}z_0)}{\ti{D}}\in SL_2(\R)}
where we have 
\equa{a+\ti{B}p_2''z_0 &\equiv A \mod \mcl{I}_1\\
\ti{B}=B+p_1y &\equiv B \mod \mcl{I}_1\\
C+p_2''(v+\ti{D}z_0)&\equiv C \mod \mcl{I}_2\\
\ti{D}=D_1+p_2Ex=D+p_2s+p_2Ex &\equiv D\mod \mcl{I}_2.}
This proves the theorem in this case.
\end{enumerate}
\item Consider the case $(B)+\mcl{I}_1=\R,(C)+\mcl{I}_2=\R$ and $(A)+\mcl{I}_1 \subsetneq \R,(D)+\mcl{I}_2\subsetneq \R$. This proof in this case is similar to the previous case.
\end{enumerate}
This completes the proof of this theorem.
\end{proof}

%%%%%%%%%%%%%%%%%%%%%%%%%%%%%%%%%%%%%%%%%%%%%%%%%%%%%%%%%%%%%%%%%%%%%%%%%%%%%%%%%%%%%%%%%
%%%%%%%%%%%%%%%%%%%%%%%%%%%%%%%%%%%%%%%%%%%%%%%%%%%%%%%%%%%%%%%%%%%%%%%%%%%%%%%%%%%%%%%%%
%%%%%%%%%%%%%%%%%%%%%%%%%%%%%%%%%%%%%%%%%%%%%%%%%%%%%%%%%%%%%%%%%%%%%%%%%%%%%%%%%%%%%%%%%

Now we prove Theorem~\ref{theorem:FullGenSurjTwoDim}. 
\begin{proof}
Let $([a:b],[c:d])\in \mbb{PF}^{1,(m^1_0,m^1_1)}_{\I_1} \times \mbb{PF}^{1,(m^2_0,m^2_1)}_{\I_2}$. Consider the matrix 
\equ{A=\begin{pmatrix}
a  & b\\
c  & d
\end{pmatrix}}
Since the image of the map is invariant under the usual $SL_{2}(\R)$ action on the right we apply elementary column operations on $A$ or right multiply $A$ by matrices in $SL_{2}(\R)$.
The ideal $\mcl{I}_1$ is contained in only a finitely many maximal ideals. By using Proposition~\ref{prop:Unital} and a suitable application of $SL_{2}(\R)$ matrix, we can assume $a$ is a unit modulo $\mcl{I}_1$
if $\mcl{I}_1\neq \R$. Again if $\mcl{I}_1\neq \R$ then by finding inverse of this element modulo $\mcl{I}_1$ and hence again by a suitable application of $SL_{2}(\R)$ matrix, the matrix $A$ can
be transformed to the following matrix $B$, where $a$ in the first row is a unit modulo $\mcl{I}_1$ and $e\in \mcl{I}_1, (a)+(e)=\R$.
\equ{B=\begin{pmatrix}
a & e\\
c & f 
\end{pmatrix}}
The ideal $\mcl{I}_2$ is contained in only a finitely many maximal ideals. If $\mcl{I}_2\neq \R$ and $c$ is not a unit $\mod \mcl{I}_2$ then, by using
Proposition~\ref{prop:Unital} and a suitable application of $SL_{2}(\R)$ matrix, we can assume the first element $c$ in the second row of $B$ is a unit modulo $\mcl{I}_2$
and the first element (also denoted by) $a$ in the first row will still remain a unit modulo $\mcl{I}_1$ if $\mcl{I}_1\neq \R$ after the transformation.
We have the following facts on the matrix $B$ now.
\begin{enumerate}
\item $a$ is a unit modulo $\mcl{I}_1$ if $\mcl{I}_1\neq \R$. 
\item $e \in \mcl{I}_1$.
\item $(a)+(e)=\R$.
\item $c$ is a unit modulo $\mcl{I}_2$ if $\mcl{I}_2\neq \R$.
\item $(c)+(f)=\R$.
\end{enumerate}
We use Theorem~\ref{theorem:FullGenSurjMainTwoDim} to find a matrix $\mattwo xyzw\in SL_2(\R)$ such that $([x:y],[z:w])=([a:e],[b:f])\in \mbb{PF}^{1,(m^1_0,m^1_1)}_{\I_1} \times \mbb{PF}^{1,(m^2_0,m^2_1)}_{\I_2}$.
This completes the proof of Theorem~\ref{theorem:FullGenSurjTwoDim} and the map is surjective.
\end{proof}

%%%%%%%%%%%%%%%%%%%%%%%%%%%%%%%%%%%%%%%%%%%%%%%%%%%%%%%%%%%%%%%%%%%%%%%%%%%%%%%%%%%%%%%%%
%%%%%%%%%%%%%%%%%%%%%%%%%%%%%%%%%%%%%%%%%%%%%%%%%%%%%%%%%%%%%%%%%%%%%%%%%%%%%%%%%%%%%%%%%
%%%%%%%%%%%%%%%%%%%%%%%%%%%%%%%%%%%%%%%%%%%%%%%%%%%%%%%%%%%%%%%%%%%%%%%%%%%%%%%%%%%%%%%%%

\section{\bf{Two Open Questions}}
\label{sec:OQ}
In this section we pose two open questions in greater generality. The first question is regarding the surjectivity of the Chinese remainder reduction map and it is stated as follows.
\begin{ques}
\label{ques:GenCRTSURJ}
Let $\R$ be a commutative ring with unity and $k,l\in \mbb{N}$. Let $\mcl{I}_i,1\leq i\leq k$ be mutually co-maximal ideals and
$\mcl{I}=\us{i=1}{\os{k}{\prod}}\mcl{I}_i$. Let $m_j\in \mbb{N},0\leq j\leq l$. When is the Chinese remainder reduction map associated to the projective space 
\equ{\mbb{PF}^{l,(m_0,m_1,\ldots,m_l)}_{\mcl{I}} \lra \mbb{PF}^{l,(m_0,m_1,\ldots,m_l)}_{\mcl{I}_1} \times \mbb{PF}^{l,(m_0,m_1,\ldots,m_l)}_{\mcl{I}_2} \times \ldots \times \mbb{PF}^{l,(m_0,m_1,\ldots,m_l)}_{\mcl{I}_k}}
surjective (and even bijective)? Or under what further general conditions
\begin{itemize}
\item on the ring $\R$,
\item on the values $m_j,0\leq j\leq l$,
\item on the co-maximal ideals $\mcl{I}_1,\ldots,\mcl{I}_k$,
\end{itemize}
is this map surjective?
\end{ques}
We have seen in Theorem~\ref{theorem:GenCRTSURJ} that the map is indeed bijective if the ideal $\mcl{I}$ is contained in only a finitely many maximal ideals with no conditions on the commutative ring $\R$ with unity. The answer to Question~\ref{ques:GenCRTSURJ} in a greater generality is not known. 

The second question is regarding the surjectivity of the map from $k\operatorname{-}$dimensional special linear group 
to the product of generalized projective spaces of $k\operatorname{-}$mutually co-maximal ideals associating the $k\operatorname{-}$rows or $k\operatorname{-}$columns.
It is stated as follows.

\begin{ques}
\label{ques:ProjHighDim}
Let $\R$ be a commutative ring with unity. Let $k \in \mbb{N}$ and $\mcl{I}_0,\mcl{I}_1,\ldots,\mcl{I}_k$ be mutually co-maximal ideals in $\R$. Let $m_j^i\in \mbb{N}, 0\leq i,j\leq k$. 
When is the map 
\equ{SL_{k+1}(\R) \lra \us{i=0}{\os{k}{\prod}}\mbb{PF}^{k,(m^i_0,m^i_1,\ldots,m^i_k)}_{\I_i}} given by
\equa{&A_{(k+1)\times (k+1)}=[a_{i,j}]_{0\leq i,j\leq k} \lra\\
&\big([a_{0,0}:a_{0,1}:\ldots: a_{0,k}],[a_{1,0}:a_{1,1}:\ldots: a_{1,k}],\ldots,[a_{k,0}:a_{k,1}:\ldots: a_{k,k}]\big)}
surjective? Or under what further general conditions
\begin{itemize}
\item on the ring $\R$,
\item on the values $m_j^i,0\leq i,j\leq k$,
\item on the co-maximal ideals $\mcl{I}_0,\mcl{I}_1,\ldots,\mcl{I}_k$,
\end{itemize}
is this map surjective?
\end{ques}
We have seen in Theorem~\ref{theorem:FullGenSurjTwoDim} that the map is indeed surjective for the projective spaces of dimension one if the two co-maximal ideals satisfy that, 
each one of the ideals is contained in only a finitely many maximal ideals. The two main Theorems~\ref{theorem:FullGenSurj},~\ref{theorem:FullGenSurjOne} answer the question under certain conditions.
The answer to Question~\ref{ques:ProjHighDim} in a greater generality is not known as illustrated in the type of examples such as Example~\ref{Example:OQ}. As stated in Question~\ref{ques:ProjHighDim}, the map is not surjective because of Example~\ref{Example:NotSurj}. 
%\section{\bf{Acknowledgments}}

\bibliographystyle{abbrv}
\def\cprime{$'$}

\end{document}